\documentclass[reqno,12pt]{amsart}
\usepackage{mathrsfs,amsmath,amssymb}
\usepackage{paralist}
\usepackage{tikz}
\usepackage{caption}
\usepackage{graphics} 
\usepackage{epsfig} 
\usepackage[colorlinks = true]{hyperref}
\hypersetup{urlcolor = red, citecolor = blue}
\usepackage[top = 1in,bottom = 1in,left = 1in,right = 1in]{geometry}
\usepackage{cite}

\usepackage{esint}

\usepackage[pagewise]{lineno}

\allowdisplaybreaks[4]

\newtheorem{thm}{Theorem}[section]

\newtheorem{lem}[thm]{Lemma}
\newtheorem{prop}[thm]{Proposition}
\theoremstyle{definition}

\newtheorem{rem}{Remark}[section]
\numberwithin{equation}{section}

\begin{document}
\title[Two-species chemotaxis system with two chemicals involving flux-limitation]
{Critical blow-up curve in a two-species chemotaxis system with two chemicals involving flux-limitation}

\author[Zeng]{Ziyue Zeng}%
\address{Z. Zeng: School of Mathematics, Southeast University, Nanjing 211189, P. R. China}
\email{ziyzzy@163.com}

\author[Li]{Yuxiang Li$^{\star}$}
\address{Y. Li: School of Mathematics, Southeast University, Nanjing 211189, P. R. China}
\email{lieyx@seu.edu.cn}


\thanks{$^{\star}$Corresponding author}

\begin{abstract}
We investigate the following two-species chemotaxis system with two chemicals involving flux-limitation
\begin{align}\tag{$\star$}
\begin{cases}
u_t = \Delta u -  \nabla \cdot \left(u(1+|\nabla v|^2)^{-\frac{p}{2}}\nabla v\right), & x \in \Omega, \ t > 0, \\ 
0 = \Delta v - \mu_w + w, \quad \mu_{w}=f_{\Omega} w, & x \in \Omega, \ t > 0, \\ 
w_t = \Delta w -  \nabla \cdot \left(w (1+|\nabla z|^2)^{-\frac{q}{2}} \nabla z\right), & x \in \Omega, \ t > 0, \\ 
0 = \Delta z - \mu_u + u, \quad \mu_{u}=f_{\Omega} u, & x \in \Omega, \ t > 0, \\
\frac{\partial u}{\partial \nu} = \frac{\partial v}{\partial \nu} = \frac{\partial w}{\partial \nu} = \frac{\partial z}{\partial \nu} = 0, & x \in \partial \Omega, \ t > 0, \\
u(x, 0) = u_0(x), \quad w(x, 0) = w_0(x), & x \in \Omega,
\end{cases}
\end{align}
where $p,q \in \mathbb{R}$ and $\Omega \subset \mathbb{R}^n$ is a smooth bounded domain. In this paper, we identify a critical blow-up curve for system ($\star$) with $n\geq 3$. If $p<\frac{n-2}{n-1}$ and $q<\frac{n-2}{n-1}$, and $\Omega=B_R(0) \subset \mathbb{R}^n$ with $n\geq 3$, there exist radially symmetric initial data such that the corresponding solution blows up in finite time; if either $p>\frac{n-2}{n-1}$ or $q>\frac{n-2}{n-1}$ with $n\geq 2$, then solutions exist globally and remain bounded.
\end{abstract}
\maketitle

\section{Introduction}
In this paper, we investigate the two-species chemotaxis system with two chemicals involving flux-limitation
\begin{align}\label{eq1.1.0}
\begin{cases}
u_t = \Delta u -  \nabla \cdot \left(uf(|\nabla v|^2)\nabla v\right), & x \in \Omega, \ t > 0, \\ 
0 = \Delta v - \mu_w + w, \quad \mu_{w}=f_{\Omega} w, & x \in \Omega, \ t > 0, \\ 
w_t = \Delta w -  \nabla \cdot \left(w g(|\nabla z|^2) \nabla z\right), & x \in \Omega, \ t > 0, \\ 
0 = \Delta z - \mu_u + u, \quad \mu_{u}=f_{\Omega} u, & x \in \Omega, \ t > 0,\\
\frac{\partial u}{\partial \nu} = \frac{\partial v}{\partial \nu} = \frac{\partial w}{\partial \nu} = \frac{\partial z}{\partial \nu} = 0, & x \in \partial \Omega, \ t > 0, \\
u(x, 0) = u_0(x), \quad w(x, 0) = w_0(x), & x \in \Omega,
\end{cases}
\end{align}
where $\Omega$ is a smooth bounded domain, $f(|\nabla v|^2)=(1+|\nabla v|^2)^{-\frac{p}{2}}$ and $g(|\nabla z|^2)=(1+|\nabla z|^2)^{-\frac{q}{2}}$. Unlike the classical Keller-Segel system, system (\ref{eq1.1.0}) exhibits a circular interaction structure. The sensitivity functions $f(|\nabla v|^2)$ and $g(|\nabla z|^2)$ describe the response to the gradients of $v$ and $z$, respectively. We refer readers to \cite{2010-MMMAS-BellomoBellouquidNietoSoler, 2022-MPRIA-Zhigun, 2020-RMI-PerthameVaucheletWang} for detailed biological backgrounds of the Keller-Segel system involving flux limitation. The goal of the present work is to identify the critical blow-up curve for system (\ref{eq1.1.0}).

The classical chemotaxis system \cite{1971-JTB-KellerSegel,1971-Jotb-KellerSegel,2002-CAMQ-PainterHillen}, involving one species and one chemical,
\begin{align}\label{eq1.1.2}
\begin{cases}
u_t=\nabla \cdot(D(u) \nabla u)-\nabla \cdot(S(u) \nabla v) ,& x \in \Omega, t>0, \\
0=\Delta v-v+u , & x \in \Omega, t>0, \\
\frac{\partial u}{\partial \nu} = \frac{\partial v}{\partial \nu} = 0, & x \in \partial \Omega, \ t > 0, \\
u(x, 0) = u_0(x), \quad v(x, 0) = v_0(x), & x \in \Omega,
\end{cases}
\end{align}
has been shown to possess the following properties.
\begin{itemize}
 \item \textbf{Critical mass phenomenon}. Let $m:=\int_{\Omega} u_0 \mathrm{d}x$. When $n=2$, $D(\xi)=1$ and $S(\xi)=\xi$, under the radially symmetric assumption, Nagai \cite{1995-AMSA-Nagai} proved that, when $m < 8\pi$, the solution remains uniformly bounded; when $m > 8\pi$, there exist initial data with small second moment $\int_{\Omega} u_0|x|^2 \mathrm{d} x$ that lead to finite-time blow-up solutions. Subsequently, Nagai \cite{2001-JIA-Nagai} extended the results to the nonradial case, and showed that, either $q \in \Omega$ and $m > 8\pi$ or $q \in \partial \Omega$ and $m > 4\pi$, if $\int_{\Omega} u_0|x-q|^2 \mathrm{d} x$ is sufficiently small, then the solution blows up in finite time. Related results for the parabolic-parabolic system (\ref{eq1.1.2}) can be found in  \cite{1997-FE-NagaiSenbaYoshida, 2001-EJAM-HorstmannWang, 
     2014--MizoguchiWinkler}.
 \item \textbf{Critical blow-up exponents phenomenon}. For system (\ref{eq1.1.2}) with $D(\xi)=(\xi+1)^p$ and $S(\xi)=\xi(\xi+1)^{q-1}$, Lankeit \cite{2020-DCDSSS-Lankeit} demonstrated that if $q-p<\frac{2}{n}$, solutions exist globally and remain bounded; if $q-p>\frac{2}{n}$, there exist radially symmetric solutions that become unbounded  either in finite time or infinite time; if $q\leq 0$, solutions are global. Similar results regarding the parabolic–parabolic system (\ref{eq1.1.2}) can be found in \cite{2010-MMAS-Winkler,2012-JDE-TaoWinkler,2012-JDE-CieslakStinner,2014-AAM-CieslakStinner,2015-JDE-CieslakStinner,2019-JDE-Winkler}.
\end{itemize}

The system (\ref{eq1.1.2}) with Jäger-Luckhaus form \cite{1992-TAMS-JaegerLuckhaus}
\begin{align}\label{eq1.1.1}
\begin{cases}
u_t=\nabla \cdot(D(u) \nabla u)-\nabla \cdot(S(u) \nabla v), & x \in \Omega, t>0, \\ 
0=\Delta v-\mu+u, \quad \mu=f_{\Omega} u, & x \in \Omega, t>0 , \\
 \frac{\partial u}{\partial \nu}=\frac{\partial v}{\partial \nu}=0, & x \in \partial \Omega, t>0, \\
u(x, 0) = u_0(x), & x \in \Omega.
\end{cases}
\end{align}
has been proved to possess the same critical mass phenomenon when $n=2$ as system (\ref{eq1.1.2}), which was demonstrated by Nagai in references \cite{1995-AMSA-Nagai, 2001-JIA-Nagai}. For system (\ref{eq1.1.1}) with $D(\xi)=(\xi+1)^p$ and $S(\xi)=\xi(\xi+1)^{q-1}$, Winkler and Djie \cite{2010-NA-WinklerDjie} showed that, if $q-p<\frac{2}{n}$, all solutions exist globally and remain bounded; if $q-p>\frac{2}{n}$ and $q>0$, under the radially symmetric assumption, there exist solutions that become unbounded in finite time.

The system with indirect signal production
\begin{align}\label{ind}
\begin{cases}
u_t=\nabla \cdot(D(u) \nabla u)-\nabla \cdot(S(u) \nabla v)-\kappa_1 u+\kappa_2 w, & x \in \Omega, t>0 \\ 
0=\Delta v-\mu_{w}(t)+w, \quad \mu_{w}(t)=f_{\Omega} w, & x \in \Omega, t>0 \\ 
w_t=\Delta w-\lambda_1 w+\lambda_2 u & x \in \Omega, t>0 \\ \frac{\partial u}{\partial \nu}=\frac{\partial v}{\partial \nu}=\frac{\partial w}{\partial \nu}=0, & x \in \partial \Omega, t>0 \\ u(x, 0)=u_0(x), \quad w(x, 0)=w_0(x), & x \in \Omega,
\end{cases}
\end{align}
where $D(\xi) \simeq \xi^{p}$ and $ S(\xi) \simeq \xi^q$ ($\xi \gg 1$), has also been shown to have two critical blow-up lines when $n \geq 3$, as identified by Tao and Winkler \cite{2025-JDE-TaoWinkler}. When $q-p>\frac{4}{n}$ and $q>\frac{2}{n}$, there exist radially symmetric initial data that lead to finite-time blow-up solutions; when $q-p<\frac{4}{n}$, the solutions are globally bounded; when $q<\frac{2}{n}$, solutions are global. They detected the blow-up by constructing subsolutions that become singular in finite time. Later, these subsolutions have also been used to determine the critical nonlinearity for blow-up in a chemotaxis system with indirect signal production in \cite{2024--Zhao}.

Considering chemotaxis systems with flux limitation,
\begin{align}\label{1.0.5}
  \left\{
  \begin{array}{ll}
    u_t =\Delta u-\nabla \cdot (uf(|\nabla v|^2) \nabla v), & x \in \Omega, t>0 , \\  
     0=\Delta v-\mu+u,  \quad \mu=f_{\Omega} u, & x \in \Omega, t>0 , \\
 \frac{\partial u}{\partial \nu}=\frac{\partial v}{\partial \nu}=0, & x \in \partial \Omega, t>0, \\
u(x, 0) = u_0(x), & x \in \Omega,
  \end{array}
  \right.
\end{align}
when $f(\xi)=\chi\xi^{\frac{p-2}{2}}$, if $p \in\big(1, \frac{n}{n-1}\big)$ ($n \geq 2$) or $p \in (1, +\infty)$ ($n=1$), Negreanu and Tello \cite{2018-JDE-NegreanuTello} obtained global bounded classical solutions. Later, Tello \cite{2022-CPDE-Tello} demonstrated that if $p \in (\frac{n}{n-1},2)$ ($n>2$), for sufficiently large $\chi$, there exist radially symmetric initial data with $\frac{1}{|\Omega|} \int_{\Omega} u_0 \mathrm{d} x>6$, such that the solutions blow up in finite time. When $f(\xi)=\chi(1+\xi)^{-\frac{p}{2}}$, Winkler \cite{2022-IUMJ-Winkler} proved that, if $0<p<\frac{n-2}{n-1}$ ($n \geq 3$), throughout a considerably large set of radially symmetric initial data, the corresponding solutions blow up in finite time; if $p>\frac{n-2}{n-1}$ ($n \geq 2$) or $p \in \mathbb{R}$ ($n=1$), all solutions are globally bounded. 

Tao and Winkler \cite{2015-DCDSSB-TaoWinkler} proposed the two-species chemotaxis system with two chemicals 
\begin{align}\label{eq1.1.0-1}
\begin{cases}
u_t = \nabla \cdot(D_1(u) \nabla u) -  \nabla \cdot \left(S_1(u) \nabla v\right), & x \in \Omega, \ t > 0, \\ 
0 = \Delta v - v + w,  & x \in \Omega, \ t > 0, \\ 
w_t = \nabla \cdot(D_2(w) \nabla w) -  \nabla \cdot \left(S_2(w) \nabla z\right), & x \in \Omega, \ t > 0, \\ 
0 = \Delta z - z + u,  & x \in \Omega, \ t > 0, \\
\frac{\partial u}{\partial \nu} = \frac{\partial v}{\partial \nu} = \frac{\partial w}{\partial \nu} = \frac{\partial z}{\partial \nu} = 0, & x \in \partial \Omega, \ t > 0, \\
u(x, 0) = u_0(x), \quad w(x, 0) = w_0(x), & x \in \Omega.
\end{cases}
\end{align}
Let $m_u:=\int_{\Omega} u_0(x) \mathrm{d} x$ and $ m_w:=\int_{\Omega} w_0(x) \mathrm{d} x$. They considered the case $D_i(\xi) \equiv 1$ and $S_i(\xi)=\xi$ and proved that, if either $n=2$ and $m_u+m_w$ lies below some threshold, or $n \geq 3$ and $\left\|u_0\right\|_{L^{\infty}(\Omega)}$, $\left\|w_0\right\|_{L^{\infty}(\Omega)}$ are sufficiently small, all solutions are globally bounded; whereas if either $n=2$ and $m_u+m_w$ is suitably large, or $n \geq 3$ and $m_u+m_w>0$ is arbitrary, there exist initial data such that the corresponding solutions blow up in finite time. Recently, the critical mass curve in two dimensions has been identified. Yu et al. \cite{2018-N-YuWangZheng} proved that, if $m_u m_w-2 \pi\left(m_u+m_w\right)>0$, then there exist finite time blow-up solutions. Yu et al. \cite{2024-NARWA-YuXueHuZhao} obtained globally bounded classical solutions, provided that $m_u m_w-2 \pi\left(m_u+m_w\right)<0$.

When $D_i(u)= (u+1)^{p_i-1}$ and $S_i(u) = u(1+u)^{q_i-1}$, Zheng \cite{2017-TMNA-Zheng} showed that solutions are globally bounded if $q_1<p_1-1+\frac{2}{n}$ and $q_2<p_2-1+\frac{2}{n}$. In the case $q_i \equiv 1$, Zhong \cite{2021-JMAA-Zhong} demonstrated that the range of $p_1$ and $p_2$ can be extended to $p_1p_2+\frac{2p_1}{n}>p_1+p-2$ or $p_1p_2+\frac{2p_2}{n}>p_1+p-2$. Recently, Zeng and Li obtained a critical blow-up curve (i.e. $q_1+q_2-\frac{4}{n}=\max\big\{(q_1-\frac{2}{n})q_2,(q_2-\frac{2}{n})q_1\big\}$ in the square $(0,\frac{4}{n}) \times (0,\frac{4}{n})$) for the system (\ref{eq1.1.0-1}) with $p_i \equiv 1$ \cite{2025--ZengLi-1} and two critical blow-up lines (i.e. $q_1-(p_1-1)=2-\frac{n}{2}$ and $q_1=1-\frac{n}{2}$) for the system (\ref{eq1.1.0-1}) with $p_2 \equiv q_2 \equiv 1$ \cite{2025--ZengLi-2}.

Motivated by the critical blow-up exponent phenomenon in system (\ref{1.0.5}), we investigate the system (\ref{eq1.1.0}) with flux-limitation and aim to find its critical blow-up curve.

$\mathbf{Main\ results.}$ Let $p,q \in \mathbb{R}$ and
\begin{align}\label{f}
f(|\nabla v|^2)=(1+|\nabla v|^2)^{-\frac{p}{2}}
\end{align}
and
\begin{align}\label{g}
g(|\nabla z|^2)=(1+|\nabla z|^2)^{-\frac{q}{2}}.
\end{align}
We assume that 
\begin{align}\label{eq1.1}
u_0, w_0 \in  W^{1,\infty}(\overline{\Omega}) \text{ are positive }.
\end{align}

The following local existence and uniqueness result is standard and a similar argument can be found in \cite{2015-DCDSSB-TaoWinkler, 2022-IUMJ-Winkler, 2025--ZengLi-1}. 

\begin{prop}\label{local}
Let $\Omega \subset \mathbb{R}^n$ $(n \geq 1)$ be a smooth bounded domain. Assume that $\left(u_0, w_0\right)$ is as in $(\ref{eq1.1})$. Then there exist $T_{\max } \in(0, \infty]$ and uniquely determined positive functions 
\begin{align*}
\begin{aligned}
& u \in C^0\left(\overline{\Omega} \times\left[0, T_{\max }\right)\right) \cap C^{2,1}\left(\overline{\Omega} \times\left(0, T_{\max }\right)\right), \\
& v \in C^{2,0}\left(\overline{\Omega} \times\left(0, T_{\max }\right)\right), \\
& w \in C^0\left(\overline{\Omega} \times\left[0, T_{\max }\right)\right) \cap C^{2,1}\left(\overline{\Omega} \times\left(0, T_{\max }\right)\right), \\
& z \in C^{2,0}\left(\overline{\Omega} \times\left(0, T_{\max }\right)\right),
\end{aligned}
\end{align*}
satisfying $\int_{\Omega} v(\cdot,t) \mathrm{~d}x=0$ and $\int_{\Omega} z(\cdot,t) \mathrm{~d}x=0$ for all $t \in (0,T_{\max})$, such that \eqref{eq1.1.0} is solved in the classical sense in $\Omega \times\left(0, T_{\max }\right)$, and the following extensibility property holds:  
\begin{align*}
\text { if } T_{\max }<\infty \text {, then } \limsup _{t \nearrow T_{\max }}\left(\|u(\cdot, t)\|_{L^{\infty}(\Omega)}+\|w(\cdot, t)\|_{L^{\infty}(\Omega)}\right)=\infty \text {. }
\end{align*}
Moreover, we have
\begin{align}\label{mass}
\int_\Omega u(\cdot,t) \mathrm{~d}x=\int_\Omega u_0 \mathrm{~d}x,\quad
\int_\Omega w(\cdot,t) \mathrm{~d}x=\int_\Omega w_0 \mathrm{~d}x, \quad
t\in(0,T_{\max}).
\end{align}

In addition, if $\Omega=B_R(0)$ for some $R>0$, and $\left(u_0, w_0\right)$ is a pair of radially symmetric functions, then $u, v, w, z$ are all radially symmetric.
\end{prop}
The first theorem demonstrate that finite-time blow-up occurs in system (\ref{eq1.1.0}).
\begin{thm}\label{thm1_1}
Let $n \geq 3$ and $\Omega=B_R(0) \subset \mathbb{R}^n$ with some $R>0$. Assume that $u_0$ and $w_0$ are radially symmetric that satisfy $(\ref{eq1.1})$. Suppose that $(\ref{f})$ and $(\ref{g})$ hold with $p,q \in \mathbb{R} $ satisfying
\begin{align}\label{pq}
p< \frac{n-2}{n-1}\quad \text{and} \quad  q< \frac{n-2}{n-1}.
\end{align}
Then, there exist functions $M_1(r), M_2(r) \in C^0([0, R])$ such that if $u_0$, $w_0$ satisfy 
\begin{align}\label{eq1.8}
\int_{B_r(0)} u_0  \mathrm{~d}x \geq M_1(r),\quad \int_{B_r(0)} w_0  \mathrm{~d}x \geq M_2(r) ,\quad  r\in(0,R),
\end{align}
the corresponding solution of \eqref{eq1.1.0} blows up in finite time.
\end{thm}
The next theorem indicates that the range in (\ref{pq}) is optimal for finite-time blow-up.
\begin{thm}\label{thm1_2}
Let $n \geq 2$ and $\Omega \subset \mathbb{R}^n$ be a smooth bounded domain. Suppose that $(\ref{f})$ and $(\ref{g})$ hold with $p,q \in \mathbb{R}$ satisfying 
\begin{align}\label{pq-global}
p>\frac{n-2}{n-1} \ \text{ or } \ q>\frac{n-2}{n-1}, 
\end{align}
Then, for any choose of $u_0$, $w_0$ complying with $(\ref{eq1.1})$, the problem $(\ref{eq1.1.0})$ possesses a unique global classical solution which is bounded in the sense that
\begin{align*}
\|u(\cdot, t)\|_{L^{\infty}(\Omega)}+\|w(\cdot, t)\|_{L^{\infty}(\Omega)} \leq C, 
\quad  t>0,
\end{align*}
with some constants $C$ independent of $t$.
\end{thm}
\begin{rem}
For $n\geq 3$, we obtain a critical blow-up curve. The results are summarized in the Figure~\ref{fig:mq_graph}.
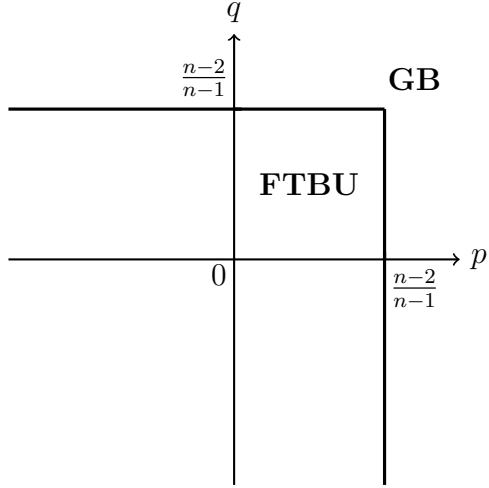
\begin{figure}[!htbp]
\begin{center}
\begin{minipage}{0.4\textwidth}
\begin{center}
\begin{tikzpicture}[scale=2.0]
\draw[->, thick] (-1.5, 0) -- (1.5, 0) node[right]{$p$};
\draw[->, thick] (0, -1.5) -- (0, 1.5) node[above]{$q$};
\draw[very thick] (1, 0) -- (1, 1);
\draw[very thick] (0, 1) -- (1, 1);
\draw[very thick] (1, -1.5) -- (1, 0.05);
\draw[very thick] (-1.5, 1) -- (0.05, 1);
\node[] at (0.5, 0.5) {\textbf{FTBU}};
\node[] at (1.2, 1.2) {\textbf{GB}};
\node[] at (-0.1, -0.1) {$0$};
\node[] at (1.2, -0.20) {$\frac{n-2}{n-1}$};
\node[] at (-0.20, 1.2) {$\frac{n-2}{n-1}$};
\end{tikzpicture}
\end{center}
\end{minipage}
\end{center}
\captionsetup{type=figure}
\captionof{figure}{
“GB”: All solutions are globally bounded. “FTBU”: There exist solutions that blow up in finite time. \\}
\label{fig:mq_graph}
\end{figure}
\end{rem}

The rest of the paper is organized as follows. In Section~\ref{comparison}, we prove a weak comparison principle. Based on this, we construct subsolutions, which have the same form as those in \cite{2025-JDE-TaoWinkler}, to detect finite-time blow-up in Section~\ref{subsolution}. Finally, we prove the global boundedness in Section~\ref{global boundedness}.

\section{A weak comparison principle}\label{comparison}

We use the mass distribution functions defined as
\begin{align}\label{eq2.1}
U(s,t):=\int^{s^\frac{1}{n}}_0 r^{n-1}u(r,t)\mathrm{~d}r\ \ \text{ and }\ \ W(s,t):=\int^{s^\frac{1}{n}}_0 r^{n-1}w(r,t)\mathrm{~d}r
\end{align}
for $s \in\left[0, R^n\right]$ and $ t \in\left[0, T_{\max }\right)$, to transform (\ref{eq1.1.0}) into the following Dirichlet parabolic system
\begin{align}
\begin{cases}
U_t=n^2 s^{2-\frac{2}{n}} U_{s s}+n U_s(W-\frac{\mu_w}{n} s) f\left(s^{\frac{2}{n}-2}(W-\frac{\mu_w}{n} s)^2\right), &s \in\left(0, R^n\right), t \in\left(0, T_{\max }\right), \\
W_t=n^2s^{2-\frac{2}{n}} W_{s s}+ n W_s(U-\frac{\mu_u}{n}s) g\left(s^{\frac{2}{n}-2}(U-\frac{\mu_u}{n} s)^2\right) ,  &s \in\left(0, R^n\right), t \in\left(0, T_{\max }\right), \\ 
U(0, t)=W(0, t)=0, \quad U\left(R^n, t\right)=\frac{\mu_uR^n}{n}, \quad W\left(R^n, t\right)=\frac{\mu_wR^n}{n}, & t \in\left(0, T_{\max }\right), \\ 
U(s, 0)=U_0(s):=\int_0^{s^{\frac{1}{n}}} \rho^{n-1} u_0(\rho) \mathrm{d} \rho, & s \in\left(0, R^n\right), \\
W(s, 0)=W_0(s):=\int_0^{s^{\frac{1}{n}}} \rho^{n-1} w_0(\rho) \mathrm{d} \rho, & s \in\left(0, R^n\right).
\end{cases}
 \end{align}
Let $T>0$. For any $\varphi, \psi\in C^1\left(\left[0, R^n\right] \times [0, T)\right)$, which satisfy $\varphi_s,\psi_s \geq 0$ on $\left(0, R^n\right) \times(0, T)$ and $\varphi(\cdot, t), \psi(\cdot, t) \in W_{l o c}^{2, \infty}\left(\left(0, R^n\right)\right)$ for all $t \in(0, T)$, we define the differential operators $\mathcal{P}$ and $\mathcal{Q}$ by
\begin{align}\label{eq2.3}
\begin{cases}
\begin{aligned}
&\mathcal{P}[\varphi, \psi](s, t):=\varphi_t-n^2 s^{2-\frac{2}{n}} \varphi_{s s}-n \varphi_s \cdot\Big(\psi-\frac{\mu^{\star} s}{n}\Big)f\left(s^{\frac{2}{n}-2}\Big(\psi-\frac{\mu^{\star} s}{n}\Big)^2\right), \\
& \mathcal{Q}[\varphi, \psi](s, t):=\psi_t-n^2 s^{2-\frac{2}{n}} \psi_{s s}-n \psi_s \cdot\Big(\varphi-\frac{\mu^{\star} s}{n}\Big) g\left(s^{\frac{2}{n}-2}\Big(\varphi-\frac{\mu^{\star} s}{n}\Big)^2\right) 
\end{aligned}
\end{cases}
\end{align}
for $t \in(0, T)$ and a.e. $s \in\left(0, R^n\right)$, where
\begin{align}\label{eq2_4}
\mu^{\star}:=\max\{\mu_u,\mu_w\}.
\end{align}

\begin{lem}
Suppose that $p,q < 1$, then $U$ and $W$, as defined in $(\ref{eq2.1})$, satisfy 
\begin{align}\label{eq2_5}
\begin{split}
\begin{cases}
\mathcal{P}[U, W](s, t) \geq 0,\quad &s\in(0,R^n), t\in(0,T_{\max}), \\
\mathcal{Q}[U, W](s, t) \geq 0, \quad &s\in(0,R^n), t\in(0,T_{\max}),\\
U(0,t)=W(0,t)=0,\quad &t\in(0,T_{\max}),\\
U(R^n,t)\geq\frac{\mu_{\star}R^n}{n},\ W(R^n,t)\geq\frac{\mu_{\star}R^n}{n},\quad &t\in(0,T_{\max}),\\
U(s,0)=\int^{s^\frac{1}{n}}_0 r^{n-1}u_0(r,t)\mathrm{d}r,\quad &s\in(0,R^n),\\
W(s,0)=\int^{s^\frac{1}{n}}_0 r^{n-1}w_0(r,t)\mathrm{d}r,\quad &s\in(0,R^n),
\end{cases}
\end{split}
\end{align}
where
\begin{align}\label{eq2_6}
\mu_{\star}:=\min\{\mu_u ,\mu_w\}.
\end{align}
\end{lem}
\begin{proof}
To compute $\mathcal{P}[U, W]$, we introduce the following notation
$$F_{W}(x):=\Big(W-\frac{x s}{n}\Big)f\left(s^{\frac{2}{n}-2}\Big(W-\frac{x s}{n}\Big)^2\right)=\Big(W-\frac{x s}{n}\Big)\left(1+s^{\frac{2}{n}-2}\Big(W-\frac{x s}{n}\Big)^2\right)^{-\frac{p}{2}}.$$
Owing to $p<1$, by a direct computation, we have
\begin{align}
\frac{d F_{W}}{dx}= & \frac{p}{2}\Big(W-\frac{xs}{n}\Big)^2\left(1+s^{\frac{2}{n}-2}\Big(W-\frac{xs}{n}\Big)^2\right)^{-\frac{p}{2}-1}
 \cdot \frac{2s^{\frac{2}{n}-1}}{n} -\frac{s}{n}\left(1+s^{\frac{2}{n}-2}\Big(W-\frac{xs}{n}\Big)^2\right)^{-\frac{p}{2}}  \notag \\ =&\frac{1}{n}\left(1+s^{\frac{2}{n}-2}\left(W-\frac{x s}{n}\right)^2\right)^{-\frac{p}{2}-1} \cdot\left(( p-1) s^{\frac{2}{n}-1}\left(W-\frac{x s}{n}\right)^2-s\right) \nonumber \\ 
\leq&  0 .
\end{align}
Thus, using (\ref{eq2.3}) and (\ref{eq2_4}), we infer that
\begin{align*}
\mathcal{P}[U, W](s, t)=&U_t-n^2 s^{2-\frac{2}{n}} U_{s s}-n U_sF_{W}(\mu^{\star}) \notag \\
\geq &U_t-n^2 s^{2-\frac{2}{n}} U_{s s}-n U_sF_{W}(\mu_w)\notag \\
=&0.
\end{align*} 
Similarly, we have $\mathcal{Q}[U, W](s, t) \geq0$ by $q<1$.
\end{proof}
The following comparison principle forms a fundamental fact for our derivation of Theorem~\ref{thm1_1}. For the proof of Theorem~\ref{thm1_1}, we define
\begin{align}\label{h}
h(x):=x(1+x^2)^{-\frac{p}{2}}, \quad x \geq 0.
\end{align}
Thus, for all $p< 1$, we have
\begin{align}\label{hprime}
h^{\prime}(x)>0, \quad x \geq 0.
\end{align}
\begin{lem}\label{lem2.2}
Let $p, q< 1$, $T>0$ and $\Omega = B_{R}(0) \subset \mathrm{R}^n$ $(n\geq 1)$. Suppose that $\underline{U}, \overline{U}, \underline{W}, \overline{W} \in C^1\left(\left[0, R^n\right] \times[0, T)\right)$ such that $\underline{U}_s, \overline{U}_s, \underline{W}_s, \overline{W}_s\geq0$ for $(s,t)\in(0,R^n) \times (0,T)$ as well as $\underline{U}(\cdot,t), \overline{U}(\cdot,t), \underline{W}(\cdot,t), \overline{W}(\cdot,t) \in W_{l o c}^{2, \infty}\left(\left(0, R^n\right)\right)$ for $t \in (0,T)$. Under the assumptions that for all $t \in(0, T)$ and a.e. $s \in\left(0, R^n\right)$,
\begin{align}\label{eq2_8}
\begin{array}{ll}
\mathcal{P}[\underline{U}, \underline{W}](s, t) \leq 0, & \mathcal{P}[\overline{U}, \overline{W}](s, t) \geq 0, \\
\mathcal{Q}[\underline{U}, \underline{W}](s, t) \leq 0, & \mathcal{Q}[\overline{U}, \overline{W}](s, t) \geq 0,
\end{array}
\end{align}
and that furthermore for $t \in[0, T)$,
\begin{align}\label{eq2_9}
\begin{array}{lll}
\underline{U}(0, t) \leq \overline{U}(0, t), & \underline{U}\left(R^n, t\right) \leq \overline{U}\left(R^n, t\right), \\
\underline{W}(0, t) \leq \overline{W}(0, t), & \underline{W}\left(R^n, t\right) \leq \overline{W}\left(R^n, t\right),
\end{array}
\end{align}
as well as for $s \in [0,R^n]$,
\begin{align}\label{eq2 9-1}
 \underline{U}(s, 0) \leq \overline{U}(s, 0),
\quad   \underline{W}(s, 0) \leq \overline{W}(s, 0),
\end{align}
it follows that
\begin{align}\label{eq2_10}
\underline{U}(s, t) \leq \overline{U}(s, t), \quad
\underline{W}(s, t) \leq \overline{W}(s, t), 
\quad  (s,t) \in\left[0, R^n\right] \times [0, T).
\end{align}
\end{lem}
\begin{proof}
Let $T_0 \in (0, T)$. For $\lambda>0$ and $\varepsilon> 0$ to be specified below, we define the functions $X(s,t)$ and $Y(s,t)$
	\begin{align}\label{XY}
		X(s, t) := \underline{U}(s, t) - \overline{U}(s, t) - \varepsilon \mathrm{e}^{\lambda t}, \quad 
        Y(s, t) := \underline{W}(s, t) - \overline{W}(s, t) - \varepsilon \mathrm{e}^{\lambda t}
	\end{align}
for $t \in [0, T_0]$ and $s \in [0, R^n]$. By \eqref{eq2_9} and \eqref{eq2 9-1}, we know that $X(0,t), Y(0,t)<0$, and $X(R^n,t), Y(R^n,t)<0$ for all $t \in [0, T_0]$, as well as $X(s,0), Y(s,0)<0$ for all $s \in [0, R^n]$.
	
We claim that
\begin{align}\label{eq2_12}
X(s, t) < 0 \text{\quad and\quad} Y(s, t) < 0, 
\quad (s,t) \in  [0, R^n] \times [0, T_0] .
\end{align}
To verify this, we assume by contradiction that \eqref{eq2_12} is false.

\textbf{Case 1}. One can find $s_X\in (0, R^n)$ and $t_X \in (0, T_0]$ such that
\begin{align}\label{eq2_13-1}
\max_{(s, t) \in [0, R^n] \times [0, t_X]} \{X(s, t), Y(s, t)\}
 = X(s_X, t_X) 
 = 0.
 \end{align}
Then, we have
	\begin{align}\label{eq2_15}
		X_t(s_X, t_X) \geq 0
	\end{align}
	and
	\begin{align}\label{eq2_16}
		X_s(s_X, t_X) = 0.
	\end{align}
	Moreover, since $X(\cdot, t_X) \in W_{\text{loc}}^{2, \infty}\left((0, R^n)\right)$, we can find a null set $N(t_X) \subset (0, R^n)$ such that $X_{ss}(s, t_X)$ exists for $s \in (0, R^n) \backslash N(t_X)$. Due to \eqref{eq2_16}, we derive that
\begin{align}\label{eq2_17}
X_s(s, t_X) = \int_{s_X}^s X_{ss}(\sigma, t_X) \, d\sigma, 
\quad  s \in (0, R^n) \backslash N(t_X).
\end{align}
As $X(\cdot, t_X)$ attains its maximum at $s_X$ by \eqref{eq2_13-1}, the identity \eqref{eq2_17} requires that there exists $\left(s_j\right)_{j \in \mathbb{N}} \subset (s_X, R^n) \backslash N(t_X)$ such that $s_j \searrow s_X$ as $j \rightarrow \infty$ and
\begin{align}\label{eq2_18}
X_{ss}(s_j, t_X) \leq 0, \quad j \in \mathbb{N},
\end{align}
for otherwise (\ref{eq2_17}) would imply that $X_s(s, t_X)>0$ for $s\in(s_X,s_{\star})$ with some $s_{\star}\in (s_X,R^n)$, which would clearly contradict (\ref{eq2_13-1}). According to \eqref{eq2_8}, \eqref{eq2_18} and the definition of $h$ in (\ref{h}), we obtain
\begin{align}\label{X}
X_t(s_j, t_X) 
=& \underline{U}_t(s_j, t_X) - \overline{U}_t(s_j, t_X) - \lambda \varepsilon \mathrm{e}^{\lambda t_X} \nonumber\\
\leq& n^2 s_j^{2 - \frac{2}{n}} X_{ss}(s_j, t_X) - \lambda \varepsilon \mathrm{e}^{\lambda t_X} + \frac{n \underline{U}_s\left(s_j, t_X\right)\left(\underline{W}\left(s_j, t_X\right)-\frac{\mu^{\star} s_j}{n}\right)}{\left(1+s_j^{\frac{2}{n}-2}\left(\underline{W}\left(s_j, t_X\right)-\frac{\mu^{\star} s_j}{n}\right)^2\right)^{\frac{p}{2}}} \nonumber \\  
&-\frac{n \overline{U}_s\left(s_j, t_X\right)\left(\overline{W}\left(s_j, t_X\right)-\frac{\mu^{\star} s_j}{n}\right)}{\left(1+s_j^{\frac{2}{n}-2}\left(\overline{W}\left(s_j, t_X\right)-\frac{\mu^{\star} s_j}{n}\right)^2\right)^{\frac{p}{2}}}\nonumber \\
\leq & \frac{n \underline{U}_s\left(s_j, t_X\right)\left(\underline{W}\left(s_j, t_X\right)-\frac{\mu^{\star} s_j}{n}\right)}{\left(1+s_j^{\frac{2}{n}-2}\left(\underline{W}\left(s_j, t_X\right)-\frac{\mu^{\star} s_j}{n}\right)^2\right)^{\frac{p}{2}}}
-\frac{n \overline{U}_s\left(s_j, t_X\right)\left(\overline{W}\left(s_j, t_X\right)-\frac{\mu^{\star} s_j}{n}\right)}{\left(1+s_j^{\frac{2}{n}-2}\left(\overline{W}\left(s_j, t_X\right)-\frac{\mu^{\star} s_j}{n}\right)^2\right)^{\frac{p}{2}}}\nonumber\\
& - \lambda \varepsilon \mathrm{e}^{\lambda t_X} \nonumber \\ 
=&n \underline{U}_s\left(s_j, t_X\right)s_j^{1-\frac{1}{n}}h(\gamma_1(s_j))
-n \overline{U}_s\left(s_j, t_X\right)s_j^{1-\frac{1}{n}}h(\gamma_2(s_j)) 
- \lambda \varepsilon \mathrm{e}^{\lambda t_X},
\end{align}
where $\gamma_1(s_j)=s_j^{\frac{1}{n}-1}\left(\underline{W}\left(s_j, t_X\right)-\frac{\mu^{\star} s_j}{n}\right) $ and $\gamma_2(s_j)=s_j^{\frac{1}{n}-1}\left(\overline{W}\left(s_j, t_X\right)-\frac{\mu^{\star} s_j}{n}\right)$. Using the Newton-Leibniz Formula, we derive that
\begin{align*}
-\frac{\mu^{\star}R}{n}\leq
\gamma_1(s_X)=s_X^{\frac{1}{n}}(\frac{\underline{W}\left(s_X, t_X\right)}{s_X}-\frac{\mu^{\star}}{n})
\leq R\|\underline{W}_s\|_{L^{\infty}(\Omega)}
\end{align*}
and
\begin{align*}
-\frac{\mu^{\star}R}{n}\leq
\gamma_2(s_X)=s_X^{\frac{1}{n}}(\frac{\overline{W}\left(s_X, t_X\right)}{s_X}-\frac{\mu^{\star}}{n})
\leq R\|\overline{W}_s\|_{L^{\infty}(\Omega)},
\end{align*}
which imply that
\begin{align*}
-\frac{\mu^{\star}R}{n}\leq
\gamma_3=\gamma_2(s_X)+\theta(\gamma_1(s_X)-\gamma_2(s_X))
\leq R(\|\underline{W}_s\|_{L^{\infty}(\Omega)}+\|\overline{W}_s\|_{L^{\infty}(\Omega)}),
\end{align*}
where $\theta\in(0,1)$. Combining this with the continuity of $h^{\prime}(x)$, we can find a positive constant $
c_1=c_1(\|\underline{W}_s\|_{L^{\infty}(\Omega)},\|\overline{W}_s\|_{L^{\infty}(\Omega)})
$
such that 
\begin{align}\label{hprime3}
h^{\prime}(\gamma_3)\leq c_1.
\end{align}
Thanks to the facts that $\underline{U}, \overline{U}, \underline{W}, \overline{W} \in C^1([0, R^n] \times (0, T))$ and $\underline{U}_s(s_X, t_X) = \overline{U}_s(s_X, t_X)$ from \eqref{eq2_16}, along with $\underline{U}_s(s_X, t_X) \geq 0$ and (\ref{hprime3}), we take $j \rightarrow \infty$ and apply the mean value theorem to see that
\begin{align*}
X_t(s_X, t_X) &\leq n \underline{U}_s(s_X,t_X)s_X^{1-\frac{1}{n}}(h(\gamma_1(s_X))-h(\gamma_2(s_X))) - \lambda \varepsilon \mathrm{e}^{\lambda t_X} \\
&=n \underline{U}_s(s_X,t_X)s_X^{1-\frac{1}{n}}h^{\prime}(\gamma_3)\left(\gamma_1(s_X)-\gamma_2(s_X)\right)- \lambda \varepsilon \mathrm{e}^{\lambda t_X}  \\
&=n \underline{U}_s(s_X,t_X)h^{\prime}(\gamma_3)\left(\underline{W}(s_X, t_X) - \overline{W}(s_X, t_X)\right)- \lambda \varepsilon \mathrm{e}^{\lambda t_X}  \\
&\leq c_1n \underline{U}_s(s_X, t_X)  \left(Y(s_X, t_X) + \varepsilon \mathrm{e}^{\lambda t_X} \right) - \lambda \varepsilon \mathrm{e}^{\lambda t_X}.
\end{align*}
We choose
\begin{align}\label{eq2_11}
\lambda \geq 2c_1n\|\underline{U}_s\|_{L^{\infty}([0, R^n] \times [0, T_0])}.
\end{align}
Since $ Y(s_X, t_X) + \varepsilon \mathrm{e}^{\lambda t_X} \leq \varepsilon \mathrm{e}^{\lambda t_X}$ by \eqref{eq2_13-1}, along with \eqref{eq2_11}, we have
\begin{align*}
X_t(s_X, t_X) \leq c_1n \underline{U}_s(s_X, t_X) \varepsilon \mathrm{e}^{\lambda t_X} - \lambda \varepsilon \mathrm{e}^{\lambda t_X} \leq -\frac{\lambda \varepsilon \mathrm{e}^{\lambda t_X}}{2},
\end{align*}
 which is absurd in view of \eqref{eq2_15}. 
 
\textbf{Case 2}. One can find $s_Y\in (0, R^n)$ and $t_Y \in (0, T_0]$ such that
\begin{align}\label{eq2_13}
\max_{(s, t) \in [0, R^n] \times [0, t_Y]} \{X(s, t), Y(s, t)\}
 = Y(s_Y, t_Y) 
 = 0,
\end{align}
which implies that $Y_t(s_Y, t_Y) \geq 0$.
Similar to the case 1, we arrive at a contradiction $Y_t(s_Y, t_Y)<0 $. 

In summary, we obtain \eqref{eq2_12}. By letting $\varepsilon \searrow 0$ and $T_0 \nearrow T$ in (\ref{XY}), we arrive at \eqref{eq2_10}.
\end{proof}

\section{ Construction of subsolutions}\label{subsolution}

The goal of this section is to prove Theorem~\ref{thm1_1}. Our approach is similar to that in \cite{2025-JDE-TaoWinkler}; however, the parameters $\alpha$ and $\beta$ used in our construction are chosen differently.

\begin{lem}\label{alphabeta-lem}
Let $n \geq 3$. Assume that $ p,q <1$ and satisfy $(\ref{pq})$. Then one can find constants $\alpha, \beta \in (0, 1-\frac{1}{n})$ and $\delta \in (0,\frac{1}{n})$ such that
\begin{align}\label{pq-2}
(1-\beta)(1-p)-\delta>0, \quad (1-\alpha)(1-q)-\delta>0
\end{align}
and
\begin{align}\label{pq-3}
(\frac{1}{n}+\beta-1)p+1-\beta-\frac{2}{n}>0, \quad (\frac{1}{n}+\alpha-1)q+1-\alpha-\frac{2}{n}>0.
\end{align}
\end{lem}
\begin{proof}
When $(\alpha, \beta, \delta)\rightarrow (0, 0, 0) $, it follows from (\ref{pq}) and $n\geq 3$ that the following limits hold: $(1-\beta)(1-p)-\delta \rightarrow 1-p>0$, $(1-\alpha)(1-q)-\delta \rightarrow 1-q>0$, $(\frac{1}{n}+\beta-1)p+1-\beta-\frac{2}{n} \rightarrow \frac{n-1}{n}(\frac{n-2}{n-1}-p)>0$ and $(\frac{1}{n}+\alpha-1)q+1-\alpha-\frac{2}{n}\rightarrow \frac{n-1}{n}(\frac{n-2}{n-1}-q)>0$. Thus, we can find $\alpha_{\star}, \beta_{\star}, \delta_{\star} \in (0,\frac{1}{2})$ such that (\ref{pq-2}) and (\ref{pq-3}) hold for $\alpha \in (0,\alpha_{\star})$, $\beta \in (0,\beta_{\star})$ and $\delta \in (0,\delta_{\star})$. 
\end{proof}
Now we specify the subsolutions that take the same form as in \cite{2025-JDE-TaoWinkler}. 
Let $\alpha, \beta \in (0,1-\frac{1}{n})$ and $\delta \in (0,\frac{2}{n})$ be taken from Lemma~\ref{alphabeta-lem}. Define $l$ by
\begin{align}\label{eq3_4}
l = \frac{\mu_{\star} R^n}{n\mathrm{e}^{\frac{1}{\mathrm{e}}}(R^n + 1)}
\end{align}
with $\mu_{\star}$ as defined in \eqref{eq2_6}. For any $y \in C^1([0, T))$ with $y(t)>\frac{1}{R^n}$ for all $t \in (0,T)$, we introduce
\begin{align}\label{phi}
	& \Phi(s, t) =
	\begin{cases}
		l y^{1 - \alpha}(t) s, & t \in [0, T), s \in \left[0, \frac{1}{y(t)}\right], \\
		l \alpha^{-\alpha} \cdot \left(s - \frac{1 - \alpha}{y(t)}\right)^\alpha, & t \in [0, T), s \in \left(\frac{1}{y(t)}, R^n\right],
	\end{cases} 
\end{align}
\begin{align}\label{psi}
	& \Psi(s, t) =
	\begin{cases}
		l y^{1 - \beta}(t) s, & t \in [0, T), s \in \left[0, \frac{1}{y(t)}\right], \\
		l \beta^{-\beta} \cdot \left(s - \frac{1 - \beta}{y(t)}\right)^\beta, & t \in [0, T), s \in \left(\frac{1}{y(t)}, R^n\right].
	\end{cases}
\end{align}
It can easily be verified that 
\begin{align*}
\Phi,\Psi \in C^1\left(\left[0, R^n\right] \times[0, T)\right) \cap C^0\left([0, T) ; W^{2, \infty}\left(\left(0, R^n\right)\right)\right)
\end{align*}
and
\begin{align*}
\Phi(\cdot,t),\Psi(\cdot,t) \in C^2\Big( [0, R^n] \setminus \Big\{\frac{1}{y(t)}\Big\} \Big), \quad \text{ for all } t \in (0,T)
\end{align*}
with
\begin{align}\label{phis}
	& \Phi_s(s, t) =
	\begin{cases}
		l y^{1 - \alpha}(t), & t \in (0, T), s \in \left(0, \frac{1}{y(t)}\right), \\
		l \alpha^{1 - \alpha} \cdot \left(s - \frac{1 - \alpha}{y(t)}\right)^{\alpha - 1}, & t \in (0, T), s \in \left(\frac{1}{y(t)}, R^n\right),
	\end{cases} 
\end{align}
\begin{align}\label{psis}
	& \Psi_s(s, t) =
	\begin{cases}
		l y^{1 - \beta}(t), & t \in (0, T), s \in \left(0, \frac{1}{y(t)}\right), \\
		l \beta^{1 - \beta} \cdot \left(s - \frac{1 - \beta}{y(t)}\right)^{\beta - 1}, & t \in (0, T), s \in \left(\frac{1}{y(t)}, R^n\right),
	\end{cases}
\end{align}
and
\begin{align}\label{phiss}
	& \Phi_{ss}(s, t) =
	\begin{cases}
		0, & t \in (0, T), s \in \left(0, \frac{1}{y(t)}\right), \\
		l \alpha^{1 - \alpha} (\alpha - 1) \cdot \left(s - \frac{1 - \alpha}{y(t)}\right)^{\alpha - 2}, & t \in (0, T), s \in \left(\frac{1}{y(t)}, R^n\right),
	\end{cases} 
\end{align}
\begin{align}\label{psiss}
	& \Psi_{ss}(s, t) =
	\begin{cases}
		0, & t \in (0, T), s \in \left(0, \frac{1}{y(t)}\right), \\
		l \beta^{1 - \beta} (\beta-1) \cdot \left(s - \frac{1 - \beta}{y(t)}\right)^{\beta - 2}, & t \in (0, T), s \in \left(\frac{1}{y(t)}, R^n\right),
	\end{cases}
\end{align}
as well as
\begin{align}\label{phit}
	& \Phi_t(s, t) =
	\begin{cases}
		l (1 - \alpha)  y^{-\alpha}(t)y^{\prime}(t) s, & t \in (0, T), s \in \left(0, \frac{1}{y(t)}\right), \\
		l \alpha^{1 - \alpha} (1 - \alpha) \cdot \left(s - \frac{1 - \alpha}{y(t)}\right)^{\alpha - 1}  \frac{y^{\prime}(t)}{y^2(t)}, & t \in (0, T), s \in \left(\frac{1}{y(t)}, R^n\right),
	\end{cases} 
\end{align}
\begin{align}\label{psit}
	& \Psi_t(s, t) =
	\begin{cases}
		l (1 - \beta) y^{-\beta}(t)y^{\prime}(t) s, & t \in (0, T), s \in \left(0, \frac{1}{y(t)}\right), \\
		l \beta^{1 - \beta} (1 - \beta) \cdot \left(s - \frac{1 - \beta}{y(t)}\right)^{\beta - 1}  \frac{y^{\prime}(t)}{y^2(t)}, & t \in (0, T), s \in \left(\frac{1}{y(t)}, R^n\right).
	\end{cases}
\end{align}
For sufficiently large $\theta>1$ to be determined later, we define
\begin{equation}\label{eq3_7}
	\begin{cases}
		\begin{array}{ll}
			\underline{U}(s, t) := \mathrm{e}^{-\theta t} \Phi(s, t), & s \in \left[0, R^n\right], t \in [0, T), \\
			\underline{W}(s, t) := \mathrm{e}^{-\theta t} \Psi(s, t), & s \in \left[0, R^n\right], t \in [0, T).
		\end{array}
	\end{cases}
\end{equation}

In the following, we aim to prove $\mathcal{P}[\underline{U}, \underline{W}]\leq 0$ and $\mathcal{Q}[\underline{U}, \underline{W}]\leq 0$ for all $t \in (0,T) \cap \left(0, \frac{1}{\theta}\right)$ and a.e, $s\in(0,R^n)$. We divide $(0,R^n)$ into three regions and begin the proof by considering the inner region $(0,\frac{1}{y(t)})$.



\begin{lem}\label{lem3.3}
Let $\Omega=B_R(0) \subset \mathbb{R}^n$ with $n \geq 3$, and let $\alpha$, $\beta$, $\delta $ be as in Lemma~\ref{alphabeta-lem}. Assume that \eqref{f} and \eqref{g} hold with $p,q <1$ satisfying $(\ref{pq})$. There exists $y_{\star}=y_{\star}(\alpha,\beta,\mu^{\star},l)>\max\{1,\frac{1}{R^n}\}$ such that if $T>0$ and a nondecreasing function $y(t) \in C^1([0, T))$ satisfies 
\begin{align}\label{y1}
\begin{cases}
\begin{aligned}
y^{\prime}(t) \leq \min\Big\{ &
\min\{2^{-1},2^{-\frac{p}{2}-1}\} n \mathrm{e}^{-2} l,\ 
2^{\frac{p}{2}-1} n \mathrm{e}^{p-2} l^{1-p} R^{-p}, \\ 
&\min\{2^{-1},2^{-\frac{q}{2}-1}\}n \mathrm{e}^{-2} l,\ 
2^{\frac{q}{2}-1} n \mathrm{e}^{q-2} l^{1-q} R^{-q}  
\Big\} y^{1+\delta}(t), \quad t \in (0,T),
\end{aligned} \\ 
y(0) > y_{\star},
\end{cases}
\end{align}
then, for arbitrary $\theta>0$, the functions $\underline{U}$ and $\underline{W}$ from $(\ref{eq3_7})$ satisfy
\begin{align*}
\mathcal{P}[\underline{U}, \underline{W}](s, t) \leq 0, \quad \mathcal{Q}[\underline{U}, \underline{W}](s, t) \leq 0, 
\end{align*}
for all $t \in (0,T) \cap \left(0, \frac{1}{\theta}\right)$ and $s \in \big(0, \frac{1}{y(t)}\big)$.
\end{lem}
\begin{proof} 
Due to $y_{\star}>\frac{1}{R^n}$ and $y^{\prime}(t) \geq 0$, we know that $\frac{1}{y(t)}<R^n$. Owing to $t \in (0,T) \cap \left(0, \frac{1}{\theta}\right)$, we have
\begin{align}\label{thetat}
\theta t<1.
\end{align}
The choices of $\alpha$, $\beta$ and $\delta$ allow us to choose $y_{\star}>\max\{1,\frac{1}{R^n}\} $ sufficiently large so that 
\begin{align}\label{y0}
{y_{\star}}^{1-\beta} > \frac{2\mu^{\star}\mathrm{e}}{nl},\quad
{y_{\star}}^ {1-\beta-\frac{1}{n}} > \frac{2\mathrm{e}}{l}, 
\end{align}
and
\begin{align}\label{y0-1}
{y_{\star}}^{1-\alpha} > \frac{2\mu^{\star}\mathrm{e}}{nl},\quad
{y_{\star}}^ {1-\alpha-\frac{1}{n}} > \frac{2\mathrm{e}}{l}.
\end{align}
In view of (\ref{thetat}) and the first restriction in \eqref{y0}, we infer that
\begin{align}\label{W}
\underline{W}- \frac{\mu^{\star} s}{n}
=&\frac{\underline{W}}{2}+\frac{\underline{W}}{2} - \frac{\mu^{\star} s}{n} \notag \\
=&\frac{\underline{W}}{2}+\frac{\mathrm{e}^{-\theta t} l y^{1 - \beta}(t) s}{2} - \frac{\mu^{\star} s}{n} \notag \\
\geq& \frac{\underline{W}}{2}+\frac{\mathrm{e}^{-1} l {y_{\star}}^{1 - \beta} s}{2} 
  -\frac{\mu^{\star} s}{n} \notag \\
\geq& \frac{\underline{W}}{2}.
\end{align}
Therefore, it follows from (\ref{thetat}), (\ref{W}), (\ref{h}) and (\ref{hprime}) that
\begin{align}\label{eq3_10}
\mathcal{P}[\underline{U}, \underline{W}](s, t) 
&= \underline{U}_t - n^2 s^{2 - \frac{2}{n}} \underline{U}_{ss} - n \underline{U}_s \cdot \Big(\underline{W} - \frac{\mu^{\star} s}{n}\Big) f\left(s^{\frac{2}{n}-2}\Big(\underline{W}-\frac{\mu^{\star} s}{n}\Big)^2\right)
\nonumber \\
&= -\theta \mathrm{e}^{-\theta t} l y^{1 - \alpha}(t) s +
\mathrm{e}^{-\theta t} l (1 - \alpha) y^{-\alpha}(t)y^{\prime}(t) s \nonumber \\
&\quad - n \mathrm{e}^{-\theta t} l y^{1 - \alpha}(t) \Big(\underline{W} - \frac{\mu^{\star} s}{n}\Big)\left(1+s^{\frac{2}{n}-2}\Big(\underline{W}-\frac{\mu^{\star} s}{n}\Big)^2\right)^{-\frac{p}{2}} \nonumber \\
&\leq \mathrm{e}^{-\theta t} l  y^{-\alpha}(t)y^{\prime}(t) s
- n \mathrm{e}^{-\theta t} l y^{1 - \alpha}(t) \Big(\underline{W} - \frac{\mu^{\star} s}{n}\Big)\left(1+s^{\frac{2}{n}-2}\Big(\underline{W}-\frac{\mu^{\star} s}{n}\Big)^2\right)^{-\frac{p}{2}}   \nonumber \\
&= \mathrm{e}^{-\theta t} l  y^{-\alpha}(t)y^{\prime}(t) s
- n \mathrm{e}^{-\theta t} l y^{1 - \alpha}(t) s^{1-\frac{1}{n}} h\left(s^{\frac{1}{n}-1}\Big(\underline{W} - \frac{\mu^{\star} s}{n}\Big)\right)\nonumber \\
& \leq  l  y^{-\alpha}(t)y^{\prime}(t) s
- n \mathrm{e}^{-1} l y^{1 - \alpha}(t) s^{1-\frac{1}{n}} h\left(\frac{s^{\frac{1}{n}-1} \underline{W}}{2} \right)\nonumber \\
& =l  y^{-\alpha}(t)y^{\prime}(t) s
- n \mathrm{e}^{-1} l y^{1 - \alpha}(t)  \frac{\underline{W}}{2}\big(1+s^{\frac{2}{n}-2}\frac{ \underline{W}^2}{4}\big)^{-\frac{p}{2}}.
\end{align}

To handle the second term on the right side of (\ref{eq3_10}), for given $t \in (0,T) \cap \left(0, \frac{1}{\theta}\right)$ , we introduce 
\begin{align}\label{D}
D(s):= \frac{s^{\frac{1}{n}-1} \underline{W}(s,t)}{2}=
		\frac{1}{2}s^{\frac{1}{n}}\mathrm{e}^{-\theta t} l y^{1 - \beta}(t), \quad s \in \big[0, \frac{1}{y(t)}\big].
\end{align}
It can be readily verified from the definition that $D(0)= 0$ and $D(s)$ is increasing in $ [0, \frac{1}{y(t)}]$. Considering the second restriction in (\ref{y0}) and $\beta \in (0,1-\frac{1}{n})$, together with (\ref{thetat}), we deduce that
\begin{align}\label{frac1y}
D\left(\frac{1}{y(t)}\right)=\frac{1}{2}\mathrm{e}^{-\theta t} l y^{1 - \beta-\frac{1}{n}}(t)
\geq \frac{l{y_{\star}}^{1 - \beta-\frac{1}{n}}}{2\mathrm{e}}>1,
\quad  t \in (0,T) \cap \big(0, \frac{1}{\theta}\big).
\end{align}
Using the continuity of $D(s)$, we infer that there exists $s_0(t) \in \big(0, \frac{1}{y(t)}\big)$ such that, 
\begin{align}\label{Din}
0\leq D(s) \leq 1, \quad \text{ for all }t \in (0,T) \cap \big(0, \frac{1}{\theta}\big)\text{ and } s \in [0,s_0(t)]
\end{align} and 
\begin{align}\label{Dout}
D(s) \geq 1, \quad \text{ for all }t \in (0,T) \cap \big(0, \frac{1}{\theta}\big)\text{ and } s \in \big(s_0(t), \frac{1}{y(t)}\big).
\end{align}

\textbf{Case 1}. $s\in (0,s_0(t)]$ and $0<p<1$. By (\ref{Din}), we have
\begin{align*}
\frac{\underline{W}}{2}\big(1+s^{\frac{2}{n}-2}\frac{ \underline{W}^2}{4}\big)^{-\frac{p}{2}}=\frac{\underline{W}}{2}\big(1+D^2(s)\big)^{-\frac{p}{2}}
\geq 2^{-\frac{p}{2}-1}\underline{W}, 
\end{align*}
for all $t \in (0,T) \cap \left(0, \frac{1}{\theta}\right)$ and $s\in [0,s_0(t)]$. Thus, using the first condition in (\ref{y1}) and $2-\beta>1+\frac{1}{n}>1+\delta$ by $\beta<1-\frac{1}{n}$ and $\delta <\frac{1}{n}$, along with $y(t)\geq 1$, it follows from (\ref{eq3_10}) that
\begin{align}\label{in-1}
\mathcal{P}[\underline{U}, \underline{W}](s, t)
&\leq l  y^{-\alpha}(t)y^{\prime}(t) s
- 2^{-\frac{p}{2}-1}n \mathrm{e}^{-1} l y^{1 - \alpha}(t) \underline{W} \nonumber \\
&= l  y^{-\alpha}(t)y^{\prime}(t) s
- 2^{-\frac{p}{2}-1}n \mathrm{e}^{-1} l y^{1 - \alpha}(t) e^{-\theta  t}l y^{1 - \beta}(t) s \nonumber \\
& \leq l y^{-\alpha}(t)y^{\prime}(t) s
-2^{-\frac{p}{2}-1}n \mathrm{e}^{-2} l^2 y^{2 - \alpha-\beta}(t)s  \nonumber \\
& = l y^{-\alpha}(t) s\big(y^{\prime}(t)-2^{-\frac{p}{2}-1}n \mathrm{e}^{-2} l y^{2 -\beta}(t)\big) \nonumber \\
& \leq l y^{-\alpha}(t) s\big(y^{\prime}(t)-2^{-\frac{p}{2}-1}n \mathrm{e}^{-2} l y^{1+\delta}(t)\big) \nonumber \\
& \leq 0,
\end{align}
for all $t \in (0,T) \cap \left(0, \frac{1}{\theta}\right)$ and $s\in [0,s_0(t)]$.

\textbf{Case 2}. $s \in \big(s_0(t), \frac{1}{y(t)}\big)$ and $0<p<1$. By (\ref{Dout}), we have
\begin{align*}
\frac{\underline{W}}{2}\big(1+s^{\frac{2}{n}-2}\frac{ \underline{W}^2}{4}\big)^{-\frac{p}{2}}
=\frac{\underline{W}}{2}\big(1+D^2(s)\big)^{-\frac{p}{2}}
\geq 2^{-\frac{p}{2}-1}\frac{\underline{W}}{D^p(s)}
=2^{\frac{p}{2}-1}s^{(1-\frac{1}{n})p} \underline{W}^{1-p}, 
\end{align*}
Relying on (\ref{pq-2}) and the second condition in (\ref{y1}), together with $y(t)\geq 1$, we deduce that
\begin{align*}
\mathcal{P}[\underline{U}, \underline{W}](s, t)
& \leq l  y^{-\alpha}(t)y^{\prime}(t) s
- 2^{\frac{p}{2}-1} n \mathrm{e}^{-1} l y^{1 - \alpha}(t)  s^{(1-\frac{1}{n})p} (\mathrm{e}^{-\theta t} l y^{1 - \beta}(t)s)^{1-p} \nonumber \\
& \leq l  y^{-\alpha}(t)y^{\prime}(t) s
- 2^{\frac{p}{2}-1} n \mathrm{e}^{p-2} l^{2-p} s^{1-\frac{p}{n}}  y^{1 - \alpha+(1 - \beta)(1-p)}(t)   \nonumber \\
& = l  y^{-\alpha}(t) s
\left(y^{\prime}(t)- 2^{\frac{p}{2}-1} n \mathrm{e}^{p-2} l^{1-p} s^{-\frac{p}{n}} y^{1 +(1 - \beta)(1-p)}(t)   \right) \nonumber \\
& \leq l  y^{-\alpha}(t)  s 
\left(y^{\prime}(t)- 2^{\frac{p}{2}-1} n \mathrm{e}^{p-2} l^{1-p} R^{-p} y^{1 +\delta}(t)  \right) \nonumber \\
& \leq 0,
\end{align*}
for all $t \in (0,T) \cap \left(0, \frac{1}{\theta}\right)$ and $s \in \big(s_0(t), \frac{1}{y(t)}\big)$.

\textbf{Case 3}. $s \in \big(0, \frac{1}{y(t)}\big)$ and $p\leq 0$. Thus, we have
\begin{align*}
\frac{\underline{W}}{2}\big(1+s^{\frac{2}{n}-2}\frac{ \underline{W}^2}{4}\big)^{-\frac{p}{2}}=\frac{\underline{W}}{2}\big(1+D^2(s)\big)^{-\frac{p}{2}}
\geq 2^{-1}\underline{W}.
\end{align*}
Similar to (\ref{in-1}) in Case 1, along with (\ref{y1}), we obtain
\begin{align*}
\mathcal{P}[\underline{U}, \underline{W}](s, t)
&\leq l  y^{-\alpha}(t)y^{\prime}(t) s
- 2^{-1}n \mathrm{e}^{-1} l y^{1 - \alpha}(t) \underline{W} \nonumber \\
& \leq l y^{-\alpha}(t) s\big(y^{\prime}(t)-2^{-1}n \mathrm{e}^{-2} l y^{1+\delta}(t)\big) \nonumber \\
& \leq 0,
\end{align*}
for all $t \in (0,T) \cap \left(0, \frac{1}{\theta}\right)$ and $s \in \big(0, \frac{1}{y(t)}\big)$.

Owing to the symmetry, we apply (\ref{y0-1}), (\ref{y1}), the second restriction in (\ref{pq-2}) to obtain $\mathcal{Q}[\underline{U}, \underline{W}](s, t) \leq 0$ for all $t \in (0,T) \cap \left(0, \frac{1}{\theta}\right)$ and $s \in \big(0, \frac{1}{y(t)}\big)$. 
\end{proof}
The following lemma demonstrates that $\mathcal{P}[\underline{U}, \underline{W}](s, t) \leq 0$ and $\mathcal{Q}[\underline{U}, \underline{W}](s, t) \leq 0$ in the intermediate region $\big(\frac{1}{y(t)},s_{\star}\big]$, provided that $s_{\star}$ is sufficiently small.

\begin{lem}\label{lem3.4}
Let $\Omega=B_R(0) \subset \mathbb{R}^n$ with $n \geq 3$, and let $\alpha$, $\beta$, $\delta $ be as in Lemma~\ref{alphabeta-lem}. Assume that \eqref{f} and \eqref{g} hold with $p,q <1$ satisfying $(\ref{pq})$. For fixed $y_{\star}$ taken from Lemma~\ref{lem3.3}, there exists a sufficiently small constant $s_{\star}=s_{\star}(\alpha,\beta,\mu^{\star},l,\delta) \in (0,R^n)$ such that if $T>0$ and a nondecreasing function $y(t) \in C^1([0, T))$ satisfies 
\begin{align}\label{y1-3}
\left\{\begin{array}{l}
y^{\prime}(t)\leq 
y^{1+\delta}(t), \ t\in (0,T), \\
y(0) > \max\{\frac{1}{s_{\star}},(1+\frac{\beta}{n-1-n\beta})\frac{1}{R^n}, y_{\star}\},
\end{array}\right.
\end{align}
then, for arbitrary $\theta>0$, the functions $\underline{U}$ and $\underline{W}$ from $(\ref{eq3_7})$ satisfy
\begin{align}\label{mid}
\mathcal{P}[\underline{U}, \underline{W}](s, t) \leq 0, \quad  \mathcal{Q}[\underline{U}, \underline{W}](s, t) \leq 0,
\end{align}
for all $t \in (0,T) \cap \left(0, \frac{1}{\theta}\right)$ and $s \in \big(\frac{1}{y(t)},s_{\star}\big]$.
\end{lem}
\begin{proof}
The interval $ (\frac{1}{y(t)},s_{\star}]$ is non-empty, owing to the fact that $y(t) \geq y(0) > \frac{1}{s_{\star}} $. Given the choices of $\alpha$, $\beta$ and $\delta$ in Lemma~\ref{alphabeta-lem}, 
we can choose $s_{\star} \in (0,R^n)$ to be sufficiently small so that 
\begin{align}\label{sstar}
s_{\star}^{1-\beta} < \frac{ln\beta^{1-\beta}}{2\mathrm{e}\mu^{\star}},\quad
s_{\star}^{1-\beta-\frac{1}{n}} < \frac{l}{2\mathrm{e}}
\end{align}
and
\begin{align}\label{sstar-1}
\frac{2\alpha^{\delta-\alpha} l}{c_1}<s_{\star}^{-\big((\frac{1}{n}+\beta-1)p+1-\beta-\delta\big)},\quad
\frac{2n^2\alpha^{\frac{2}{n}-\alpha-1}l}{c_1}<s_{\star}^{-\big((\frac{1}{n}+\beta-1)p+1-\beta-\frac{2}{n}\big)} ,
\end{align}
\begin{align}\label{sstar-2}
s_{\star}^{1-\alpha} < \frac{ln\alpha^{1-\alpha}}{2\mathrm{e}\mu^{\star}},\quad
s_{\star}^{1-\alpha-\frac{1}{n}} < \frac{l}{2\mathrm{e}},
\end{align}
as well as
\begin{align}\label{sstar-3}
\frac{2\beta^{\delta-\beta} l}{c_2}<s_{\star}^{-\big((\frac{1}{n}+\alpha-1)q+1-\alpha-\delta\big)},\quad
\frac{2n^2\beta^{\frac{2}{n}-\beta-1}l}{c_2}<s_{\star}^{-\big((\frac{1}{n}+\alpha-1)q+1-\alpha-\frac{2}{n}\big)},
\end{align}
where 
\begin{align}\label{c2}
c_1=\min\{2^{p-1}\alpha^{-(1-\frac{1}{n})p},2^{\frac{p}{2}-1}\}c_{\star}^{\beta(1-p)}n \mathrm{e}^{p-2} \alpha^{1-\alpha}l^{2-p}\beta^{-\beta(1-p)} 
\end{align}
and $$c_2=\min\{2^{q-1}\beta^{-(1-\frac{1}{n})q},2^{\frac{q}{2}-1}\}{c_{\star \star}}^{\alpha(1-q)}n \mathrm{e}^{q-2} \beta^{1-\beta}l^{2-q}\alpha^{-\alpha(1-q)} $$ 
with $c_{\star}=\min\{\frac{\beta}{\alpha},1\}$ and $c_{\star \star}=\min\{\frac{\alpha}{\beta},1\}$. According to the definitions of $\underline{U}$, $\underline{W}$, $\mathcal{P}$ and $h$ defined in (\ref{h}), along with $\theta t<1$ by (\ref{thetat}), we have
\begin{align}\label{p-1} 
\mathcal{P}[\underline{U}, \underline{W}](s, t)
=& \underline{U}_t - n^2 s^{2 - \frac{2}{n}} \underline{U}_{ss} - n \underline{U}_s \cdot \Big(\underline{W} - \frac{\mu^{\star} s}{n}\Big) f\left(s^{\frac{2}{n}-2}\Big(\underline{W}-\frac{\mu^{\star} s}{n}\Big)^2\right)
\nonumber \\
=&-\theta \mathrm{e}^{-\theta t}\alpha^{-\alpha}l\Big(s-\frac{1-\alpha}{y(t)}\Big)^\alpha
+\mathrm{e}^{-\theta t}\alpha^{1-\alpha} l(1-\alpha)\Big(s-\frac{1-\alpha}{y(t)}\Big)^{\alpha-1}\frac{y^{\prime}(t)}{y^{2}(t)} \nonumber \\
&+\mathrm{e}^{-\theta t}n^2s^{2-\frac{2}{n}}\alpha^{1-\alpha} l (1-\alpha)\Big(s-\frac{1-\alpha}{y(t)}\Big)^{\alpha-2}\nonumber \\
& -n\mathrm{e}^{-\theta t}\alpha^{1-\alpha}l\Big(s-\frac{1-\alpha}{y(t)}\Big)^{\alpha-1} s^{1-\frac{1}{n}}
h\left(s^{\frac{1}{n}-1}\Big(\underline{W}-\frac{\mu^{\star} s}{n}\Big)\right)\nonumber \\
\leq& \alpha^{1-\alpha} l\Big(s-\frac{1-\alpha}{y(t)}\Big)^{\alpha-1}\cdot y^{\delta-1}(t)
+n^2s^{2-\frac{2}{n}}\alpha^{1-\alpha}l \Big(s-\frac{1-\alpha}{y(t)}\Big)^{\alpha-2} \nonumber \\
& -n\mathrm{e}^{-\theta t}\alpha^{1-\alpha}l\Big(s-\frac{1-\alpha}{y(t)}\Big)^{\alpha-1}
 s^{1-\frac{1}{n}}
h\left(s^{\frac{1}{n}-1}\Big(\underline{W}-\frac{\mu^{\star} s}{n}\Big)\right),
\end{align}
for all $t \in (0,T) \cap \big(0, \frac{1}{\theta}\big)$ and $s \in \big(\frac{1}{y(t)},s_{\star}\big]$. 
Due to $\delta \in (0,\frac{1}{n})$, for all $s>\frac{1}{y(t)}$, we obtain
\begin{align}\label{s1y}
y^{\delta-1}(t)<{\alpha}^{\delta-1}\Big(s-\frac{1-\alpha}{y(t)}\Big)^{1-\delta},\
\alpha s<s-\frac{1-\alpha}{y(t)}, \
\beta s<s-\frac{1-\beta}{y(t)}.
\end{align}
Employing the first two inequalities in (\ref{s1y}), we estimate the first two terms on the right-hand side of (\ref{p-1}), and thus derive that
\begin{align}\label{p-2}
 \mathcal{P}[\underline{U}, \underline{W}](s, t) 
\leq & \alpha^{\delta-\alpha} l\Big(s-\frac{1-\alpha}{y(t)}\Big)^{\alpha-\delta}
+n^2\alpha^{\frac{2}{n}-\alpha-1}l \Big(s-\frac{1-\alpha}{y(t)}\Big)^{\alpha-\frac{2}{n}}\nonumber\\
& -n\mathrm{e}^{-\theta t}\alpha^{1-\alpha}l\Big(s-\frac{1-\alpha}{y(t)}\Big)^{\alpha-1}
 s^{1-\frac{1}{n}}
h\left(s^{\frac{1}{n}-1}\Big(\underline{W}-\frac{\mu^{\star} s}{n}\Big)\right).
\end{align}

We estimate the last term on the right-hand side of (\ref{p-2}) and define
\begin{align*}
I:=n\mathrm{e}^{-\theta t}\alpha^{1-\alpha}l\Big(s-\frac{1-\alpha}{y(t)}\Big)^{\alpha-1}
 s^{1-\frac{1}{n}}
h\left(s^{\frac{1}{n}-1}\Big(\underline{W}-\frac{\mu^{\star} s}{n}\Big)\right) .
\end{align*}
The combination of the third inequality in (\ref{s1y}) and the first restriction in (\ref{sstar}), along with (\ref{thetat}), allows us to conclude that
\begin{align}\label{W2}
\frac{\underline{W}}{2}-\frac{\mu^{\star}s}{n}
=& \frac{1}{2} \mathrm{e}^{-\theta t}\beta^{-\beta}l  \left(s - \frac{1 - \beta}{y(t)}\right)^\beta-\frac{\mu^{\star}s}{n} \nonumber \\
\geq &\frac{l \left(s - \frac{1 - \beta}{y(t)}\right)^\beta}{2\mathrm{e}\beta^{\beta} }-\frac{\mu^{\star}\left(s - \frac{1 - \beta}{y(t)}\right)}{n\beta} \nonumber \\
=&\frac{\mu^{\star}}{n\beta}\left(s - \frac{1 - \beta}{y(t)}\right)^\beta
\left(\frac{ln\beta^{1-\beta}}{2\mathrm{e}\mu^{\star}}-\left(s - \frac{1 - \beta}{y(t)}\right)^{1-\beta}\right) \nonumber \\
\geq & \frac{\mu^{\star}}{n\beta}\left(s - \frac{1 - \beta}{y(t)}\right)^\beta
\left(\frac{ln\beta^{1-\beta}}{2\mathrm{e}\mu^{\star}}-s_{\star}^{1-\beta}\right) \nonumber \\
\geq & 0.
\end{align}
For given $t \in (0,T) \cap \left(0, \frac{1}{\theta}\right)$, we define
\begin{align*}
D(s):= s^{\frac{1}{n}-1}\frac{ \underline{W}(s,t)}{2}=\frac{1}{2}s^{\frac{1}{n}-1}\mathrm{e}^{-\theta t}l \beta^{-\beta} \cdot \left(s - \frac{1 - \beta}{y(t)}\right)^\beta, \quad  s \in \big(\frac{1}{y(t)},s_{\star}\big].
\end{align*}
We apply the third inequality in (\ref{s1y}) and the second restriction in (\ref{sstar}) to deduce that
\begin{align}\label{Dsstar}
D(s_{\star})
= \frac{1}{2}s_{\star}^{\frac{1}{n}-1} \mathrm{e}^{-\theta t} l \beta^{-\beta} \cdot \left(s_{\star} - \frac{1 - \beta}{y(t)}\right)^\beta
\geq      \frac{s_{\star}^{\frac{1}{n}-1}l \beta^{-\beta} (\beta s_{\star})^{\beta}}{ 2\mathrm{e}}
> 1.
\end{align}
Using $y(t)\geq y(0)>(1+\frac{\beta}{n-1-n\beta})\frac{1}{R^n}$, we infer that $ \frac{(1-\beta)(n-1)}{(n-1-n\beta)y(t)}<R^n$. Due to $0<\beta<1-\frac{1}{n}$, we know that $D(s)$ is increasing on $(\frac{1}{y(t)}, \frac{(1-\beta)(n-1)}{(n-1-n\beta)y(t)})$, and decreasing on $( \frac{(1-\beta)(n-1)}{(n-1-n\beta)y(t)},R^n)$. Combining the monotonicity of $D(s)$ with (\ref{Dsstar}) and (\ref{frac1y}) by $y(0)>y_{\star}$, we infer that 
\begin{align*}
D(s) \geq 1, \quad \text{ for all } t \in (0,T) \cap \big(0, \frac{1}{\theta}\big) \text{ and } s \in \big(\frac{1}{y(t)},s_{\star}\big].
\end{align*}
Therefore, according to (\ref{W2}) and the monotonicity of $h(x)$ defined in (\ref{h}), for any $p<1$, we have
\begin{align*}
h\Big(s^{\frac{1}{n}-1}\big(\underline{W}-\frac{\mu^{\star} s}{n}\big)\Big) 
\geq& h\Big(s^{\frac{1}{n}-1}\frac{\underline{W}}{2}\Big) =s^{\frac{1}{n}-1}\frac{\underline{W}}{2}\Big(1+\big(s^{\frac{1}{n}-1}\frac{\underline{W}}{2}\big)^2\Big)^{-\frac{p}{2}}\notag \\
=&D(s)\big(1+D^2(s)\big)^{-\frac{p}{2}}
\geq \min\{1,2^{-\frac{p}{2}}\} D^{1-p}
= \min\{1,2^{-\frac{p}{2}}\}\Big(s^{\frac{1}{n}-1}\frac{\underline{W}}{2}\Big)^{1-p}.
\end{align*}
Thus, by the definition of $I$, we have
\begin{align*}
I \geq & \min\{2^{p-1},2^{\frac{p}{2}-1}\}n\mathrm{e}^{-1}\alpha^{1-\alpha}l\Big(s-\frac{1-\alpha}{y(t)}\Big)^{\alpha-1}
 s^{(1-\frac{1}{n})p}\underline{W}^{1-p} \nonumber \\
= &\min\{2^{p-1},2^{\frac{p}{2}-1}\}n\mathrm{e}^{-1}\alpha^{1-\alpha}l\Big(s-\frac{1-\alpha}{y(t)}\Big)^{\alpha-1}
 s^{(1-\frac{1}{n})p}\Big(\mathrm{e}^{-\theta t} l \beta^{-\beta}  \big(s - \frac{1 - \beta}{y(t)}\big)^\beta \Big)^{1-p}.
\end{align*}
When $p\leq 0$, using the second inequality in (\ref{s1y}), we know that
\begin{align}\label{p<0-1}
s^{(1-\frac{1}{n})p}\geq\alpha^{-(1-\frac{1}{n})p}\Big(s-\frac{1-\alpha}{y(t)}\Big)^{(1-\frac{1}{n})p}.
\end{align}
For any $p<1$, thanks to $s - \frac{1 - \beta}{y(t)}>c_{\star}\big(s - \frac{1 - \alpha}{y(t)}\big)$ with $c_{\star}=\min\{\frac{\beta}{\alpha},1\}$, together with (\ref{p<0-1}), we obtain that
\begin{align*}
I \geq & \min\{2^{p-1},2^{\frac{p}{2}-1}\} c_{\star}^{\beta(1-p)}n \mathrm{e}^{p-2} \alpha^{1-\alpha}l^{2-p}\beta^{-\beta(1-p)} \Big(s-\frac{1-\alpha}{y(t)}\Big)^{\alpha-1+\beta(1-p)}  s^{(1-\frac{1}{n})p} \nonumber \\
\geq& \min\{2^{p-1}\alpha^{-(1-\frac{1}{n})p},2^{\frac{p}{2}-1}\}c_{\star}^{\beta(1-p)}n \mathrm{e}^{p-2} \alpha^{1-\alpha}l^{2-p}\beta^{-\beta(1-p)} \Big(s-\frac{1-\alpha}{y(t)}\Big)^{\alpha-1+\beta(1-p)+(1-\frac{1}{n})p} \nonumber \\
=&c_1\Big(s-\frac{1-\alpha}{y(t)}\Big)^{(1-\frac{1}{n}-\beta)p+\alpha+\beta-1}
\end{align*}
with $c_1 $ defined in (\ref{c2}). Thus, inserting this into (\ref{p-2}), and noticing that (\ref{sstar-1}), we show that
\begin{align*}
\mathcal{P}[\underline{U}, \underline{W}](s, t)
\leq & \alpha^{\delta-\alpha} l\Big(s-\frac{1-\alpha}{y(t)}\Big)^{\alpha-\delta}
+n^2\alpha^{\frac{2}{n}-\alpha-1}l \Big(s-\frac{1-\alpha}{y(t)}\Big)^{\alpha-\frac{2}{n}} \nonumber \\
&-c_1 \Big(s-\frac{1-\alpha}{y(t)}\Big)^{(1-\frac{1}{n}-\beta)p+\alpha+\beta-1}  \nonumber \\
= & \frac{c_1}{2}\Big(s-\frac{1-\alpha}{y(t)}\Big)^{\alpha-\delta} \left(\frac{2\alpha^{\delta-\alpha} l}{c_1}-\Big(s-\frac{1-\alpha}{y(t)}\Big)^{-\big((\frac{1}{n}+\beta-1)p+1-\beta-\delta\big)} \right) \nonumber \\
&+\frac{c_1}{2} \Big(s-\frac{1-\alpha}{y(t)}\Big)^{\alpha-\frac{2}{n}} 
\left( \frac{2n^2\alpha^{\frac{2}{n}-\alpha-1}l}{c_1}-\Big(s-\frac{1-\alpha}{y(t)}\Big)^{-\big((\frac{1}{n}+\beta-1)p+1-\beta-\frac{2}{n}\big)} \right) \nonumber \\
\leq & \frac{c_1}{2}\Big(s-\frac{1-\alpha}{y(t)}\Big)^{\alpha-\delta} \left(\frac{2\alpha^{\delta-\alpha} l}{c_1}-s_{\star}^{-\big((\frac{1}{n}+\beta-1)p+1-\beta-\delta\big)}\right) \nonumber \\
&+\frac{c_1}{2} \Big(s-\frac{1-\alpha}{y(t)}\Big)^{\alpha-\frac{2}{n}} 
\left( \frac{2n^2\alpha^{\frac{2}{n}-\alpha-1}l}{c_1}-s_{\star}^{-\big((\frac{1}{n}+\beta-1)p+1-\beta-\frac{2}{n}\big)} \right) \nonumber \\
\leq & 0, 
\end{align*}
for all $t \in (0,T) \cap \left(0, \frac{1}{\theta}\right)$ and $s \in \big(\frac{1}{y(t)},s_{\star}\big]$. A similar argument, based on the symmetry, the second condition in (\ref{pq-3}), and the smallness assumptions (\ref{sstar-2}) and (\ref{sstar-3}) on $s_{\star}$, shows that $\mathcal{Q}[\underline{U}, \underline{W}](s, t) \leq 0$ for all $t \in (0,T) \cap \left(0, \frac{1}{\theta}\right)$ and $s \in \big(\frac{1}{y(t)},s_{\star}\big]$. We complete our proof.
\end{proof}
The following lemma shows that, for sufficiently large $\theta$, $\mathcal{P}[\underline{U}, \underline{W}](s, t) \leq0$ and $\mathcal{Q}[\underline{U}, \underline{W}](s, t)\leq 0$ hold in the outer region $\left(s_{\star}, R^n\right)$.
\begin{lem}\label{lem3.5}
Let $\Omega=B_R(0) \subset \mathbb{R}^n$ with $n \geq 3$, and let $\alpha$, $\beta$, $\delta $ be as in Lemma~\ref{alphabeta-lem}. Assume that \eqref{f} and \eqref{g} hold with $p,q <1$ satisfying $(\ref{pq})$. For fixed $s_{\star}$ taken from Lemma~\ref{lem3.4}, there exists a sufficiently large constant $\theta^{\star}=\theta^{\star}(\alpha,\beta,\mu^{\star},l,\delta)$ such that if $T>0$ and a nondecreasing function $y(t) \in C^1([0, T))$ satisfies 
\begin{align}\label{y1-2}
\left\{\begin{array}{l}
y^{\prime}(t)\leq 
y^{1+\delta}(t), \ t\in (0,T), \\
y(0) > \frac{1}{s_{\star}},
\end{array}\right.
\end{align}
then, whenever $\theta>\theta^{\star}$, the functions $\underline{U}$ and $\underline{W}$ from $(\ref{eq3_7})$ satisfy
\begin{align*}
\mathcal{P}[\underline{U}, \underline{W}](s, t) \leq 0, \quad \mathcal{Q}[\underline{U}, \underline{W}](s, t) \leq 0,
\end{align*}
for all $t \in (0,T) \cap \left(0, \frac{1}{\theta}\right)$ and $s \in \left(s_{\star}, R^n\right)$.
\end{lem}
\begin{proof}
We fix $\theta^{\star}$ large enough such that 
\begin{align}\label{theta-1}
\frac{l {s_{\star}}^\alpha}{\mathrm{e}} \theta^{\star}
\geq ls^{\alpha-\delta}_{\star} + \frac{n^2 l R^{2n - 2} {s_{\star}}^{\alpha - 2}}{\alpha}  + c^{\star}\mu^{\star}l s_{\star}^{\alpha-1} R^n
\end{align}
and
\begin{align}\label{theta-2}
\frac{l {s_{\star}}^\beta}{\mathrm{e}} \theta^{\star}
\geq ls^{\beta-\delta}_{\star} + \frac{n^2 l R^{2n - 2} {s_{\star}}^{\beta - 2}}{\beta}  + c^{\star \star}\mu^{\star}l s_{\star}^{\beta-1} R^n,
\end{align}
where $c^{\star}=\min\{1,\big(1+s_{\star}^{\frac{2}{n}-2}(l {\beta}^{-\beta}R^{n\beta})^2\big)^{-\frac{p}{2}}\}$ and $c^{\star \star}=\min\{1,\big(1+s_{\star}^{\frac{2}{n}-2}(l {\alpha}^{-\alpha}R^{n\alpha})^2\big)^{-\frac{q}{2}}\}$.
By $s_{\star}>\frac{1}{y(t)}$ and $\delta \in (0,\frac{1}{n})$, we deduce that 
\begin{align}\label{s7}
R^{n}>s-\frac{1-\alpha}{y(t)}>s_{\star}-\frac{1-\alpha}{y(t)}>\alpha s_{\star},
\end{align}
and
\begin{align}\label{s8}
R^{n}>s-\frac{1-\beta}{y(t)}>s_{\star}-\frac{1-\beta}{y(t)}>\beta s_{\star},
\end{align}
as well as
\begin{align}\label{s9}
y^{\delta-1}(t)<s^{1-\delta}_{\star}.
\end{align}
According to (\ref{s8}) and the definition of $\underline{W}$, we know that
\begin{align*}
s^{\frac{2}{n}-2}\left(\underline{W}-\frac{\mu^{\star} s}{n}\right)^2
\leq s_{\star}^{\frac{2}{n}-2}\underline{W}^2
\leq s_{\star}^{\frac{2}{n}-2}(l {\beta}^{-\beta}R^{n\beta})^2.
\end{align*}
Thus, we have
\begin{align*}
f\left(s^{\frac{2}{n}-2}\Big(\underline{W}-\frac{\mu^{\star} s}{n}\Big)^2\right)
=\left(1+s^{\frac{2}{n}-2}\Big(\underline{W}-\frac{\mu^{\star} s}{n}\Big)^2\right)^{-\frac{p}{2}}
\leq c^{\star},
\end{align*}
where $c^{\star}=\min\{1,\big(1+s_{\star}^{\frac{2}{n}-2}(l {\beta}^{-\beta}R^{n\beta})^2\big)^{-\frac{p}{2}}\}$.
Using (\ref{theta-1}), (\ref{s7}) and (\ref{s9}), we infer that
\begin{align*}
\begin{aligned}
\mathcal{P}[\underline{U}, \underline{W}](s, t)
=& \underline{U}_t - n^2 s^{2 - \frac{2}{n}} \underline{U}_{ss} - n \underline{U}_s \cdot \Big(\underline{W} - \frac{\mu^{\star} s}{n}\Big) f\left(s^{\frac{2}{n}-2}\Big(\underline{W}-\frac{\mu^{\star} s}{n}\Big)^2\right)
\nonumber \\
\leq & \underline{U}_t - n^2 s^{2 - \frac{2}{n}} \underline{U}_{ss}+ c^{\star} n \underline{U}_s \frac{\mu^{\star} s}{n} \nonumber \\
\leq & -\theta \mathrm{e}^{-\theta t} l \alpha^{-\alpha} \Big(s-\frac{1-\alpha}{y(t)}\Big) ^{\alpha} 
+ \mathrm{e}^{-\theta t} l \alpha^{1-\alpha} (1-\alpha) \Big(s-\frac{1-\alpha}{y(t)}\Big)^{\alpha-1}y^{\delta-1}(t) \\ \notag
&+ n^2 s^{2-\frac{2}{n}} \mathrm{e}^{-\theta t}l(1-\alpha) \alpha^{1-\alpha}  \Big(s-\frac{1-\alpha}{y(t)}\Big)^{\alpha-2}
+ c^{\star}\mu^{\star} \mathrm{e}^{-\theta t} l \alpha^{1 - \alpha} \cdot \left(s - \frac{1 - \alpha}{y(t)}\right)^{\alpha - 1} s\\ \notag
\leq & -\frac{l\theta^{\star} {s_{\star}}^\alpha}{\mathrm{e}}+  ls^{\alpha-\delta}_{\star} + \frac{n^2 l R^{2n - 2} {s_{\star}}^{\alpha - 2}}{\alpha}  + c^{\star}\mu^{\star}l s_{\star}^{\alpha-1} R^n\\ \notag
\leq & 0,
\end{aligned}
\end{align*}
for all $t \in (0,T) \cap \left(0, \frac{1}{\theta}\right)$ and $s \in \left(s_{\star}, R^n\right)$. Similarly, from (\ref{theta-2}), (\ref{s8}) and (\ref{s9}), we find that
\begin{align*}
\begin{aligned}
\mathcal{Q}[\underline{U}, \underline{W}](s, t) 
=& \underline{W}_t - n^2 s^{2 - \frac{2}{n}} \underline{W}_{ss} - n \underline{W}_s \cdot \Big(\underline{U} - \frac{\mu^{\star} s}{n}\Big) g\left(s^{\frac{2}{n}-2}\Big(\underline{U}-\frac{\mu^{\star} s}{n}\Big)^2\right)
\nonumber \\
\leq & -\frac{l\theta {s_{\star}}^\beta}{\mathrm{e}}+  ls^{\beta-\delta}_{\star} + \frac{n^2 l R^{2n - 2} {s_{\star}}^{\beta - 2}}{\beta}  + c^{\star \star}\mu^{\star}l s_{\star}^{\beta-1} R^n\\ \notag
\leq & 0, \text{ for all } t \in (0,T) \cap \big(0, \frac{1}{\theta}\big) \text{ and } s \in \left(s_{\star}, R^n\right),
\end{aligned}
\end{align*}
where $c^{\star \star}=\min\{1,\big(1+s_{\star}^{\frac{2}{n}-2}(l {\alpha}^{-\alpha}R^{n\alpha})^2\big)^{-\frac{q}{2}}\}$. We complete our proof.
\end{proof}

\begin{proof}
[\textbf{Proof of Theorem \ref{thm1_1}}.] 
Using (\ref{eq3_4}) and the definition of $\underline{U}$, along with $\alpha^{-\alpha}=\mathrm{e}^{-\alpha\mathrm{ln}\alpha} \leq \mathrm{e}^{\frac{1}{\mathrm{e}}}$, we have
	\begin{align}\label{uRn}
		\begin{aligned}
			\underline{U}(R^n, t) &= \mathrm{e}^{-\theta t} \alpha^{-\alpha} l \left( R^n - \frac{1 - \alpha}{y(t)} \right)^{\alpha} 
			\leq \alpha^{-\alpha} l R^{n\alpha} 
			= \alpha^{-\alpha} R^{n\alpha} \frac{\mu_{\star} R^n}{n \mathrm{e}^{\frac{1}{\mathrm{e}}}(R^n + 1)} \\
			&\leq \frac{\mu_{\star} R^n}{n} \cdot \frac{R^{n\alpha}}{R^n + 1} 
            \leq \frac{\mu_{\star} R^n}{n}
            \leq U(R^n, t).
		\end{aligned}
	\end{align}
In (\ref{eq1.8}), we take 
\begin{align*}
M_1(r) = \omega_n  \underline{U}(r^n, 0), \quad \quad M_2(r) = \omega_n \underline{W}(r^n, 0),
 \quad  r \in [0, R],
\end{align*}
where $\omega_n$ is the surface area of the unit sphere. Then, we deduce that
\begin{align}\label{u0}
\underline{U}(s, 0)=\frac{1}{\omega_n}M_1(s^{\frac{1}{n}})
\leq \frac{1}{\omega_n}\int_{B_{s^{\frac{1}{n}}}(0)} u_0  \mathrm{~d}x=U(s,0).
\end{align}
Similarly, we have
\begin{align}\label{wRnw0}
\underline{W}(R^n, t) \leq  W(R^n, t) \ \text{ and } \ \ \underline{W}(s, 0) \leq W(s,0).
\end{align}

Take $\alpha$, $\beta$ and $\delta$ as in Lemma~\ref{alphabeta-lem}, $s_{\star}$ as in Lemma~\ref{lem3.4} and $\theta^{\star}$ as in Lemma~\ref{lem3.5}. For given $\theta>\theta^{\star}$ and $y_{\star} $ from Lemma~\ref{lem3.3}, we define
\begin{align*}
\gamma=\min\Big\{&1,\min\{2^{-1},2^{-\frac{p}{2}-1}\} n \mathrm{e}^{-2} l,\ 
2^{\frac{p}{2}-1} n \mathrm{e}^{p-2} l^{1-p} R^{-p}, \\ 
&\min\{2^{-1},2^{-\frac{q}{2}-1}\}n \mathrm{e}^{-2} l,\ 
2^{\frac{q}{2}-1} n \mathrm{e}^{q-2} l^{1-q} R^{-q}   \Big\}
\end{align*}
and
\begin{align}\label{y0-4}
y_0>\max\big\{1,\frac{1}{s^{\star}},(1+\frac{\beta}{n-1-n\beta})\frac{1}{R^n},y_{\star},\big(\frac{\theta}{\gamma\delta}\big)^{\frac{1}{\delta}}\big\}.
\end{align}
Let $y(t)$ be the blow-up solution of the following ODE:
\begin{align}\label{y}
\left\{\begin{array}{l}
y^{\prime}(t)= \gamma
y^{1+\delta}(t), \ t\in (0,T), \\
y(0)=y_0,
\end{array}\right.
\end{align}
with 
\begin{align}\label{T}
T=\frac{1}{\gamma \delta} y_0^{-\delta}<\frac{1}{\theta}.
\end{align}
Then, $y^{\prime}(t) \geqslant 0$ and $y(t)\rightarrow +\infty$ as $t \nearrow T$. Our choice of $y(t)$ satisfying (\ref{y0-4})-(\ref{T}) meet the requirements in Lemmas~\ref{lem3.3}-\ref{lem3.5}. Recalling to Lemmas~\ref{lem3.3}-\ref{lem3.5} and (\ref{T}), we have
\begin{align*}
\mathcal{P}[\underline{U}, \underline{W}](s, t) \leq 0, \quad \mathcal{Q}[\underline{U}, \underline{W}](s, t) \leq 0, \quad
(s,t) \in(0,R^n) \setminus \left\{ \frac{1}{y(t)} \right\} \times (0,T).
\end{align*}
Combining this with (\ref{uRn}), (\ref{u0}) and (\ref{wRnw0}), along with $\underline{U}(0,t)=U(0,t)=\underline{W}(0,t)=W(0,t)=0$, we deduce that \begin{align*}
\underline{U}(s, t) \leq \overline{U}(s, t), \quad
\underline{W}(s, t) \leq \overline{W}(s, t),\quad
(s,t) \in(0,R^n) \setminus \left\{ \frac{1}{y(t)} \right\} \times (0,T).
\end{align*}
 Thanks to $U(0,t)=\underline{U}(0,t)=0$, we obtain
\begin{align}\label{us=0}
\frac{1}{n} \cdot u(0, t) = U_s(0, t) \geq \underline{U}_s(0, t) = \mathrm{e}^{-\theta t} \cdot l y^{1 - \alpha}(t) \geq \frac{l}{\mathrm{e}} \cdot y^{1 - \alpha}(t) \rightarrow +\infty \quad \text{as } t \nearrow T.
\end{align}
Similarly, we conclude that
\begin{align*}
\frac{1}{n} \cdot w(0, t)  \geq \frac{l}{\mathrm{e}} \cdot y^{1 - \beta}(t) \rightarrow +\infty \quad \text{as } t \nearrow T.
\end{align*}
Combining this with (\ref{us=0}) yields $T_{\max}\leq T<\infty$, which leads to a contradiction with the assumption $T_{\max}=\infty$.

\end{proof}


\section{Global boundedness}\label{global boundedness}
In this section, we are devoted to proving Theorem~\ref{thm1_2} by applying the method in \cite{2022-IUMJ-Winkler}. Using the well-known $W^{1,p}$ regularity theory \cite{1973-JMSJ-BrezisStrauss} to the second equation in (\ref{eq1.1.0}), we derive the following lemma.
\begin{lem}\label{nablav}
For all $k\in [1,\frac{n}{n-1})$ with $n \geq 2$, there exists a constant $C=C(k)>0$ such that
\begin{align*}
\|\nabla v(\cdot,t)\|_{L^{ k}(\Omega)} \leq C\|w(\cdot,t)\|_{L^1(\Omega)}, \quad t\in(0,T_{\max}).
\end{align*}
\end{lem}
\begin{proof}
[\textbf{Proof of Theorem \ref{thm1_2}}.] We need to consider two cases. 

\textbf{Case 1}. $q\in \mathbb{R}$ and $p>\frac{n-2}{n-1}$. Owing to $p>\frac{n-2}{n-1}$, we can infer that $n(1-p)<\frac{n}{n-1}$. Thus, we can fix $k \in [1,\frac{n}{n-1})$ such that $k>n(1-p)$, which guarantees that $\frac{1-p}{k}<\frac{1}{n}$. Thus, for $n\geq 2$, we can select $r>n$ such that
\begin{align}\label{r}
\frac{1-p}{k}<\frac{1}{r}<\frac{1}{n} \leq 1\text{.}
\end{align}
Due to the known smoothing properties of the Neumann heat semigroup $\left(e^{t \Delta}\right)_{t \geq 0}$ on $\Omega$ $($\cite{2016-DCDS-FujieItoWinklerYokota}$)$, we can find positive constants $\lambda$ and $c_{1}$ such that, for all $\varphi \in C^1\left(\bar{\Omega}\right)$ such that $\varphi \cdot \nu=0$ on $\partial \Omega$,
\begin{align}\label{heat}
\left\|e^{t \Delta} \nabla \cdot \varphi\right\|_{L^{\infty}(\Omega)} 
\leq c_1 t^{-\frac{1}{2}-\frac{n}{2 r}} e^{-\lambda t}\|\varphi\|_{L^r(\Omega)},
\quad t >0.
\end{align}
We employ a variation-of-constants representation associated with the first equation in (\ref{eq1.1.0}), along with (\ref{heat}) and the maximum principle, to see that
\begin{align}\label{uinfty}
&\|u(\cdot, t)\|_{L^{\infty}(\Omega)} \nonumber \\
& =\left\|\mathrm{e}^{t \Delta} u_0-\int_0^t \mathrm{e}^{(t-s) \Delta} \nabla \cdot\left\{u(\cdot, s)f\left(|\nabla v(\cdot, s)|^2\right) \nabla v(\cdot, s)\right\} \mathrm{d} s\right\|_{L^{\infty}(\Omega)} \nonumber \\
& \leq \left\|\mathrm{e}^{t \Delta} u_0\right\|_{L^{\infty}(\Omega)}
+c_2 \int_0^t\left\|\mathrm{e}^{(t-s) \Delta} \nabla \cdot\left\{u(\cdot, s)
f\left(|\nabla v(\cdot, s)|^2\right) \nabla v(\cdot, s)\right\}\right\|_{L^{\infty}(\Omega)} \mathrm{d}s \nonumber \\
& \leq \left\|u_0\right\|_{L^{\infty}(\Omega)}+c_1 c_2 \int_0^t(t-s)^{-\frac{1}{2}-\frac{n}{2 r}} \mathrm{e}^{-\lambda(t-s)}
\left\|u(\cdot, s)f\left(|\nabla v(\cdot, s)|^2\right) \nabla v(\cdot, s)\right\|_{L^r(\Omega)} \mathrm{d} s.
\end{align}
Writing $M(T):=\sup _{t \in(0, T)}\|u(\cdot, t)\|_{L^{\infty}(\Omega)}$ for any $T \in\left(0, T_{\max }\right)$. Without loss of generality, we assume that $M(T)>1$. For the case $p \geq 1$, using Hölder's inequality, along with (\ref{mass}), one can find a positive constant $c_3$ such that
\begin{align}\label{p>1}
\left\|u(\cdot, s)f\left(|\nabla v(\cdot, s)|^2\right) \nabla v(\cdot, s)\right\|_{L^r(\Omega)}  
=&\left\|u(\cdot, s)\left(1+|\nabla v(\cdot, s)|^2\right)^{-\frac{p}{2}} \nabla v(\cdot, s)\right\|_{L^r(\Omega)}  \nonumber \\
\leq & \|u(\cdot, s)\|_{L^r(\Omega)}  \nonumber \\
\leq & \|u(\cdot, s)\|_{L^{\infty}(\Omega)}^{a_1} \|u(\cdot, s)\|_{L^{1}(\Omega)}^{1-a_1} \nonumber \\
\leq & c_3 M^{a_1}(T)
\end{align}
with $a_1=1-\frac{1}{r} \in (0,1)$ by (\ref{r}).  
For the case $\frac{n-2}{n-1}<p<1$, using Lemma~\ref{nablav}, similar to (\ref{p>1}), we obtain
\begin{align}\label{p<1}
\left\|u(\cdot, s)f\left(|\nabla v(\cdot, s)|^2\right) \nabla v(\cdot, s)\right\|_{L^r(\Omega)}
\leq & \left\|u(\cdot, s)|\nabla v(\cdot, s)|^{1-p}\right\|_{L^r(\Omega)} \nonumber \\
\leq & \|u(\cdot, s)\|_{L^{\frac{rk}{k-r(1-p)}}(\Omega)} \|\nabla v(\cdot, s)\|_{L^{k}(\Omega)}^{1-p} \nonumber \\
\leq & \|u(\cdot, s)\|_{L^{\infty}(\Omega)}^{a_2} \|u(\cdot, s)\|_{L^{1}(\Omega)}^{1-a_2} \|\nabla v(\cdot, s)\|_{L^{k}(\Omega)}^{1-p} \nonumber \\
\leq & c_3 M^{a_2}(T),
\end{align}
where $a_2=1-\frac{1}{r}+\frac{1-p}{k} \in (0,1)$ by (\ref{r}) and $p<1$. Let $a=\max\{a_1,a_2\}<1$. Inserting (\ref{p>1}) and (\ref{p<1}) into (\ref{uinfty}), along with $r>n$, there exists a constant $c_4>0$ such that
\begin{align*}
\|u(\cdot, t)\|_{L^{\infty}(\Omega)} 
& \leq \left\|u_0\right\|_{L^{\infty}(\Omega)}+c_1c_2c_3 M^{a}(T) \int_0^t(t-s)^{-\frac{1}{2}-\frac{n}{2 r}} \mathrm{e}^{-\lambda(t-s)}\mathrm{d}s \nonumber \\
&\leq c_4+c_4 M^{a}(T) ,
\quad t \in (0,T).
\end{align*}
Therefore, we have $M(T) \leq c_4+c_4 M^a(T)$ for all $T \in\left(0, T_{\max }\right)$, which implies that $\|u(\cdot, t)\|_{L^{\infty}(\Omega)} $ $ \leq \max\{1,\left(2 c_4\right)^{\frac{1}{1-a}}\} $ for all $t \in\left(0, T_{\max }\right)$ by $a<1$. 

Based on the regularity results for linear elliptic equations, and applying them to the fourth equation in (\ref{eq1.1.0}), we can find positive constants $c_5$ and $c_6$ such that $$\|\nabla z(\cdot, t)\|_{L^{\infty}(\Omega)} \leq c_5 \|u(\cdot, t)\|_{L^{\infty}(\Omega)} \leq c_6, \quad t \in\left(0, T_{\max }\right).$$
Therefore, by $g\left(|\nabla z|^2\right)=(1+|\nabla z|^2)^{-\frac{q}{2}}\leq c_7$ for $q\in  \mathbb{R}$, we have
\begin{align*}
\left\|w(\cdot, s)g\left(|\nabla z|^2\right) \nabla z(\cdot, s)\right\|_{L^{\gamma}(\Omega)}
\leq c_7\|w(\cdot, s)\|_{L^{\gamma}(\Omega)} \|\nabla z(\cdot, s)\|_{L^{\infty}(\Omega)}
\leq c_6c_7 \|w(\cdot, s)\|_{L^{\gamma}(\Omega)}.
\end{align*}
Thus, again using the variation-of-constants representation and (\ref{heat}), for any $\gamma>n$, one can find constants $c_8, c_9>0$ such that
\begin{align*}
\|w(\cdot, t)\|_{L^{\infty}(\Omega)} 
& =\left\|\mathrm{e}^{t \Delta} w_0-\int_0^t \mathrm{e}^{(t-s) \Delta} \nabla \cdot\left\{w(\cdot, s)g\left(|\nabla z(\cdot, s)|^2\right) \nabla z(\cdot, s)\right\} \mathrm{d} s\right\|_{L^{\infty}(\Omega)} \nonumber \\
& \leq \left\|\mathrm{e}^{t \Delta} w_0\right\|_{L^{\infty}(\Omega)}
+c_8 \int_0^t\left\|\mathrm{e}^{(t-s) \Delta} \nabla \cdot\left\{w(\cdot, s)
g\left(|\nabla z(\cdot, s)|^2\right) \nabla z(\cdot, s)\right\}\right\|_{L^{\infty}(\Omega)} \mathrm{d} s \nonumber \\
& \leq \left\|w_0\right\|_{L^{\infty}(\Omega)}
+c_8 \int_0^t(t-s)^{-\frac{1}{2}-\frac{n}{2 \gamma}} \mathrm{e}^{-\lambda(t-s)}
\left\|w(\cdot, s)g\left(|\nabla z(\cdot, s)|^2\right) \nabla z(\cdot, s)\right\|_{L^{\gamma}(\Omega)} \mathrm{d}  s \nonumber \\
& \leq \left\|w_0\right\|_{L^{\infty}(\Omega)}
+c_6c_7c_8  \sup _{t \in(0, T)}\|w(\cdot, t)\|_{L^{\gamma}(\Omega)} 
\int_0^t(t-s)^{-\frac{1}{2}-\frac{n}{2 \gamma}} \mathrm{e}^{-\lambda(t-s)} \mathrm{d}  s \nonumber \\
& \leq \left\|w_0\right\|_{L^{\infty}(\Omega)}
+c_6 c_7c_8 \|w(\cdot, t)\|_{L^{1}(\Omega)}^{\frac{1}{\gamma}} \sup _{t \in(0, T)} \|w(\cdot, t)\|_{L^{\infty}(\Omega)}^{1-\frac{1}{\gamma}}   \int_0^t(t-s)^{-\frac{1}{2}-\frac{n}{2 \gamma}} \mathrm{e}^{-\lambda(t-s)} \mathrm{d}  s \nonumber \\
& \leq c_9+c_9 \sup _{t \in(0, T)} \|w(\cdot, t)\|_{L^{\infty}(\Omega)}^{1-\frac{1}{\gamma}},
\quad t\in (0,T).
\end{align*}
Similarly, we can obtain $\|w(\cdot, t)\|_{L^{\infty}(\Omega)} \leq \max \{1, (2c_9)^{\gamma}\}$ for all $t\in(0,T_{\max})$.

\textbf{Case 2}. $p\in \mathbb{R}$ and $q>\frac{n-2}{n-1}$. Due to the symmetry of system (\ref{eq1.1.0}), similar to the Case 1, we omit the proof.
\end{proof}
\vskip 3mm

\noindent\textbf{Data availability} The manuscript has no associated data.

\vskip 3mm
\noindent {\large\textbf{Declarations}}
\vskip 2mm 

\noindent\textbf{Conflict of interest} On behalf of all authors, the corresponding author states that there is no conflict of interest.

\noindent\textbf{Acknowledgments}

The authors thank the anonymous referees for their helpful comments and suggestions, which greatly improve the presentation of our paper. This paper is partially supported by National Natural Science Foundation of China (No. 12271092, No. 11671079) and the Jiangsu Provincial Scientific Research Center of Applied Mathematics (No. BK20233002).

\end{document}